\documentclass{amsart}
\usepackage[T1]{fontenc}
\usepackage{latexsym,amssymb,amsmath}

\usepackage{amsthm}
\usepackage{longtable}
\usepackage{epsfig}
\usepackage{hhline}
\usepackage{epic}

  \newcommand{\Res}{\operatorname{Res}}
  \newcommand{\Ind}{\operatorname{Ind}}
  \newcommand{\Fix}{\operatorname{Fix}}
  \newcommand{\Tr}{\operatorname{Tr}}
  
  \newcommand{\cyc}[1]{\langle\,#1\,\rangle}
  \newcommand{\Br}{\operatorname{IBr}}
  \newcommand{\Irr}{\operatorname{Irr}}
  \newcommand{\Aut}{\operatorname{Aut}}
  \newcommand{\sym}{\mathfrak{S}}
  
  \newcommand{\C}{\mathbb{C}}

  \newcommand{\N}{\mathbb{N}}
  
  \newcommand{\Z}{\mathbb{Z}}
  \newcommand{\A}{\mathfrak A}
  \newcommand{\B}{\cal B}

\newcommand{\cal}[1]{\mathcal{#1}}

\newtheorem{theorem}{Theorem}[section] 
\newtheorem{lemma}[theorem]{Lemma}     

\newtheorem{proposition}[theorem]{Proposition}

\newtheorem{remark}[theorem]{Remark}
\newtheorem{convention}[theorem]{Convention}

\title[A basic set for the alternating group]
  {A basic set for the alternating group}

\author{Olivier Brunat}

\address{Ruhr-Universit\"at Bochum\\
Fakult\"at f\"ur Mathematik\\
Raum NA 2/33\\
D-44780 Bochum\\}
\email{Olivier.Brunat@ruhr-uni-bochum.de}

\author{Jean-Baptiste Gramain}

\address{\'Ecole Polytechnique F\'ed\'erale de Lausanne\\
Institut de g\'eom\'etrie, alg\`ebre et topologie\\
B\^atiment de Chimie (BCH)\\
CH-1015 Lausanne\\}
\email{jean-baptiste.gramain@epfl.ch}

\begin{document}

\begin{abstract}
This article concerns the $p$-basic set existence problem in the
representation theory of finite groups. We show that, for any odd prime $p$, the alternating
group $\A_n$ has a $p$-basic set.
More precisely, we prove that the symmetric group $\sym_n$
has a $p$-basic set 
with some additional properties, 
allowing us to deduce a $p$-basic set for $\A_n$. 
Our main tool is the generalized perfect
isometries introduced by K\"ulshammer, Olsson and Robinson.
As a consequence we obtain some results on the decomposition number of
$\A_n$.
\end{abstract}

\maketitle

\section{Introduction}\label{intro}

Let $G$ be a finite group and $p$ be a prime. We are interested in
the representation theory of $G$ over a field of characteristic $p$.
In the framework of Brauer's modular representation theory, we can
construct a decomposition map $d_p:\cal R_0(G)\rightarrow \cal R_p(G)$
from the Grothendieck group of finite-dimensional representations
of $G$ in characteristic $0$ to the analogous Grothendieck group
in characteristic $p$. A fundamental result of Brauer's shows that
this map is surjective. This motivates the following definition. A set
$B$ of irreducible representations of $G$ in characteristic $0$ is
called a $p$-{\emph{basic set}} of $G$ if the images of the classes
of the elements in $B$ under $d_p$ form a basis of $\cal R_p(G)$. Note
that it is not clear if $p$-basic sets always exist. The
$p$-basic sets are powerful tools to compute the $p$-decomposition
matrix of $G$. For example, Hiss in~\cite{HissG2}
and Geck in~\cite{geck3D4} used them to compute the $p$-decomposition
matrices of $G_2(q)$ and $SU_3(q^2)$ respectively, when $p$ is
not the defining characteristic. It is an open question however if such
$p$-basic sets exist in general. The answer is known to be positive in
the following situations: for $p$-soluble groups (this is a direct
consequence of the result of Fong-Swan in~\cite[X.2.1]{Feit}), for
finite groups of Lie type in the non-defining characteristics under some
additional hypotheses (cf \cite{GH}, \cite{GeckBS} and \cite{GeckBS3}),
for the finite general linear groups $\operatorname{GL}_2(q)$,
$\operatorname{GL}_3(q)$ and $\operatorname{GL}_4(q)$ in the defining
characteristic (cf \cite{Br1}), and also for the symmetric group
$\sym_n$ (cf \cite[6.3.60]{James-Kerber}).

This question is still open for the alternating group $\A_n$. In the
present work, we prove that $p$-basic sets do exist for $\A_n$, whenever
$p$ is an odd prime. For every subset $B$ of irreducible characters
of $\sym_n$, we denote by $B_{\A_n}$ the set of all irreducible
constituents of any of the $\Res_{\A_n}^{\sym_n}(\chi)$'s ($\chi\in\
B$). We say that $B_{\A_n}$ is the {\emph{restriction}} of $B$ to
$\A_n$. Our approach is to find a $p$-basic set $B$ of $\sym_n$
which restricts to a $p$-basic set of $\A_n$. If $B$ is a $p$-basic
set of $\sym_n$, the set $B_{\A_n}$ is in general not a $p$-basic
set of $\A_n$. For example, the $p$-basic set of $\sym_n$ described
in~\cite[6.3.60] {James-Kerber} does not restrict to a $p$-basic set of
$\A_n$.

In \cite[\S4]{GeckBS}, Geck gives conditions on $B$ ensuring that
$B_{\A_n}$ is a $p$-basic set. More precisely, he proves that, for
any finite group $G$ with a normal subgroup $H$ such that $G/H$ is
abelian, and for any prime $p$ not dividing $|G/H|$, if there is a $\cal
R_0(G/H)$-stable $p$-basic set $B$ of $G$ such that the set $d_p(B_H)$
is a system of generators  of $\cal R_p(H)$, then $B_H$ is a
$p$-basic set of $H$. To prove that $\A_n$ has a $p$-basic set, we will
construct a $p$-basic set of $\sym_n$ which satisfies Geck's conditions.

In order to state our main result, we first recall some results
and notations. The conjugacy classes and irreducible characters
of $\sym_n$ are canonically labelled by the partitions of $n$
(cf \cite[2.1.11]{James-Kerber}). For $\lambda$ a partition of
$n$ (written $\lambda \vdash n$), we write $\chi_{\lambda}$ the
corresponding irreducible character of $\sym_n$. We denote by
$\varepsilon$ the sign character of $\sym_n$. This is the linear
character of $\sym_n$ which extends the non-trivial character of
$\Irr(\sym_n/\A_n)$. Moreover if $\chi_{\lambda}\in\Irr(\sym_n)$, then
$\varepsilon\chi_{\lambda}=\chi_{\lambda^*}$, where $\lambda^*$ denotes
the conjugate partition of $\lambda$. If $\lambda=\lambda^*$, then we
say that $\lambda$ is {\emph{self-conjugate}}.

Now let $p$ be a prime. A partition $\lambda$ of $n$ is said to be
$p${\emph{-regular}} if it has no parts divisible by $p$ (because
$\lambda$ then labels a conjugacy class of $p$-regular elements of
$\sym_n$). Note that, sometimes, one also calls $p${\emph{-regular}}
any partition of $n$ with no part repeated more than $p-1$ times (cf
\cite[6.1]{James-Kerber}). With this terminology, the set of $p$-regular
partitions of $n$ labels the $p$-basic set of $\sym_n$ in \cite[6.3.60]
{James-Kerber}. However, as we mentionned above, this basic set does not
restrict to a basic set of $\A_n$, and we will never in this paper use
$p$-regular in this sense.

For $\lambda \vdash n$, we denote by $\gamma(\lambda)$ its $p$-core
and by $\alpha_{\lambda}=(\lambda^{(1)},\ldots,\lambda^{(p)})$ its
$p$-quotient. In the following, the $i$-th part $\lambda^{(i)}$
of the $p$-quotient $\alpha_{\lambda}$ will also be denoted
by $\alpha_{\lambda}^{i}$. Note that the $p$-quotient of
$\lambda$ is not uniquely defined, but depends on a convention
(or a choice of origin). This is described more precisely
in~\S\ref{conjugationabacus}. However, one can show (cf Lemma
\ref{conjugationabacus} and Convention~\ref{convention}) that, for a
certain choice of origin, the $p$-quotient of $\lambda^*$ is given by
$\alpha_{\lambda^*}=(\lambda^{(p)^*},\ldots,\lambda^{(1)^*})$ (and this
independantly on the choice of $\lambda$, or even of $n$).

Finally, to each self-conjugate partition
$\lambda=(\lambda_1,\ldots,\lambda_k)$ of $n$, we associate the
partition $\overline{\lambda}$ of $n$ given by
$$\overline{\lambda}=(2\lambda_1-1,2\lambda_2-3,\ldots,2\lambda_k
-(2k-1)).$$
Note that the parts of $\overline{\lambda}$ are given by the lengths
of the {\emph{diagonal hooks}} in the Young diagram of $\lambda$ (i.e.
those hooks whose top left corner lies on the diagonal). Our main result
is:

\begin{theorem}\label{BSn}
We keep the notation as above. Let $p$ be an odd prime and let $n\in\N$.
We set
$$\Lambda_{\emptyset}=\{\lambda\vdash n\ |\ \alpha_{\lambda}^{(p+1)/2}
=\emptyset\}.$$ 
Then the subset $\B_{\emptyset}=\{\chi_{\lambda}\ |\
\lambda\in\Lambda_{\emptyset}\}$ is a $p$-basic set of the symmetric
group $\sym_n$. Furthermore, $\B_{\emptyset}$ satisfies the following properties
\begin{enumerate} 
\item If $\chi_{\lambda}\in \B_{\emptyset}$, then $\chi_{\lambda^*}\in
\B_{\emptyset}$.
\item If $\lambda=\lambda^*$, then $\chi_{\lambda}\in \B_{\emptyset}$ if
and only if the partition $\overline{\lambda}$ is $p$-regular.
\end{enumerate}
\end{theorem}

Note that the additional properties (1) and (2) are direct consequences
of Convention~\ref{convention} (cf Lemma \ref{preg}). The hardest part
is to prove that $\B_{\emptyset}$ is indeed a $p$-basic set.
As a consequence we will prove that $\cal B_{\emptyset,\A_n}$, the
restriction of the $p$-basic set $\cal B_{\emptyset}$ obtained in the
above theorem, is a $p$-basic set of $\A_n$ (see Theorem~\ref{BSAn}).

\smallskip

In order to do this, we will need a more general concept of basic sets,
which goes as follows. Let $\cal C$ denote a union of conjugacy classes
of $G$. For every class function $\varphi$ on $G$, we define a class
function $\varphi^{\cal C}$ by letting $\varphi^{\cal C}(g)=\varphi(g)$
if $g\in \cal C$, and $\varphi^{\cal C}(g)=0$ otherwise. Let $b$ be a
subset of $\Irr(G)$. A $\cal C${\emph{-basic set}} of $b$ is a subset
$B\subseteq b$, such that the family $B^{\cal C}=\{\chi^{\cal C} |
\chi\in B\}$ is a $\Z$-basis of the $\Z$-module generated by $b^{\cal
C}=\{\chi^{\cal C} | \chi\in b\}$. If $b$ is the whole of $\Irr(G)$,
then a $\cal C$-basic set of $b$ is just called $\cal C$-basic set of
$G$ (or for $G$). Moreover, we note that if $p$ is a prime, and if we
take $\cal C$ to be the set of $p$-regular elements of $G$, then we
write $\cal C=p$-reg, and $\cal C$-basic sets are the same as $p$-basic
sets as defined before.

We now present the notions of {\emph{generalized blocks}} and
{\emph{generalized perfect isometry}} introduced by K\"ulshammer,
Olsson and Robinson in \cite{KOR}. Let $\cal C$ be a union of
conjugacy classes of a finite group $G$. For any complex-valued class
functions $\alpha, \, \beta$ of $G$, we let $$\cyc{\alpha,\beta}_{\cal
C}=\frac{1}{|G|}\sum_{g\in \cal C}\alpha(g)\overline{\beta(g)}.$$ If
$\cyc{\alpha,\beta}_{\cal C}=0$, then $\alpha$ and $\beta$ are said
to be {\emph{orthogonal across}} ${\cal C}$. Note that, if $\chi, \,
\varphi \in \Irr(G)$, then $\cyc{\chi,\varphi}_{\cal C}=\cyc{\chi^{\cal
C},\varphi}_G=\cyc{\chi^{\cal C},\varphi^{\cal C}}_G$, where $\chi^{\cal
C}$ and $\varphi^{\cal C}$ are the class functions we defined before,
and $\cyc{ \; , \; }_G$ is the ordinary scalar product on class
functions of $G$.

We define the $\cal C${\emph{-blocks}} of $G$ to be the minimal
subsets of $\Irr(G)$ subject to being orthogonal across ${\cal C}$.
In particular, it is well-known that, if ${\cal C}$ is the set of
$p$-regular elements of $G$, then the ${\cal C}$-blocks are just the
$p$-blocks of modular representation theory. Note also that the ${\cal
C}$-blocks are always the same as the $(G \setminus {\cal C})$-blocks.

Finally, note that, for $\chi \in \Irr(G)$, $\{ \chi \}$ is a ${\cal
C}$-block if and only if $\chi = \chi^{\cal C}$ (if $1 \in {\cal C}$) or
$\chi = \chi^{G \setminus \cal C}$ (if $1 \not \in {\cal C}$).

For $b \subseteq \Irr(G)$, we write $(b,\cal C)$ to indicate that
$b$ is a $\cal C$-block. Following~\cite{KOR}, we say that there
is a {\emph{generalized perfect isometry}} between two blocks
$(b,\cal C)$ and $(b',\cal C')$ of $G$ and $H$ respectively if there
exists a bijection ${\cal I} \colon b \rightarrow b'$ and signs
$\{ \eta(\chi) , \, \chi \in b \}$ such that, for all $\chi, \,
\varphi \in b$, $\cyc{\chi,\varphi}_{\cal C}=\cyc{\eta(\chi)\cal
I(\chi),\eta(\varphi)\cal I(\varphi)}_{\cal C'}$. Writing ${\cal
I}_{\eta}(\chi)$ for $\eta(\chi)\cal I(\chi)$ ($\chi \in b$), we will
say, with a slight abuse of notation, that ${\cal I}_{\eta} \colon
(b, {\cal C}) \rightarrow (b', {\cal C}')$ is a generalized perfect
isometry. Note that ${\cal I}_{\eta}$ induces an $\C$-vector space
isomorphism between the subspaces of $\C \Irr(G)$ and $\C \Irr(H)$
generated by $b$ and $b'$ respectively, which we will also denote by
${\cal I}_{\eta}$.

Our main tool to construct a $p$-basic set of $\sym_n$ as in
Theorem~\ref{BSn} will be the following fact (cf Proposition
\ref{BSisometrie}): with the notations above, if ${\cal I}_{\eta} \colon
(b,\cal C)\rightarrow (b',\cal C')$ is a generalized perfect isometry,
then $B \subseteq b$ is a $\cal C$-basic set of $b$ if and only if ${\cal
I}(B) \subseteq b'$ is a $\cal C'$-basic set of $b'$.

\medskip
For any odd prime $p$ and any positive integer $w$, we set
$$G_{p,w}=(\Z_p\rtimes\Z_{p-1})\wr \sym_w.$$ We can now explain more
explicitly our strategy to prove Theorem~\ref{BSn}.
\begin{itemize}
\item
We obtain a generalized perfect isometry
$${\cal I}_{\eta} \colon (\Irr(G_{p,w}),\cal
C_{\emptyset})\rightarrow (b,p\textrm{-reg}),
$$ where $b$ is any $p$-block of $\sym_n$ of weight $w>0$, and $\cal
C_{\emptyset}$ is a union of conjugacy classes described in~\S\ref{wp}.
\item The signature~$\varepsilon$ of $\sym_n$ permutes the $p$-blocks of
$\sym_n$. We construct a $\cal C_{\emptyset}$-basic set $B_{\emptyset}$
of $G_{p,w}$ such that, if $b$ is an $\varepsilon$-stable $p$-block
of $\sym_n$ of weight $w>0$, then ${\cal I}(B_{\emptyset})$ is an
$\varepsilon$-stable $p$-basic set of $b$.
\item We use this to obtain an $\varepsilon$-stable $p$-basic set of
$\sym_n$, and check that Theorem~\ref{BSn}(2) holds.\\
\end{itemize}

The article is organized as follows. In Section~\ref{part2}, we
show that a $\cal C$-basic set of a $\cal C$-block $b$ is mapped
via a generalized perfect isometry ${\cal I}_{\eta} \colon (b,\cal
C)\rightarrow (b',\cal C')$ to a $\cal C'$-basic set of $b'$. In
Section~\ref{part3}, we present a generalized perfect isometry,
based on the work of the second author in \cite{JB}, between any
$p$-block of $\sym_n$ of weight $w>0$ and $\Irr(G_{p,w})$, with respect
to a certain union of conjugacy classes $\cal C_{\emptyset}$. In
Section~\ref{part4}, we describe a $\cal C_{\emptyset}$-basic set of
$G_{p,w}$. Finally in Section~\ref{part5}, we prove Theorem~\ref{BSn}.
We will show that this $p$-basic set of $\Irr(\sym_n)$ satisfies Geck's
conditions. As a consequence, we prove in Section~\ref{part5} that,
for any odd prime $p$, there is a $p$-basic of $\A_n$. Finally, in
Section~\ref{consequence}, we give some results on the decomposition
numbers of $\A_n$.

\section{Results on basic sets and generalized perfect
isometries}\label{part2}
We keep the notation as in Section~\ref{intro}.
\begin{lemma}\label{blocks}
Let $G$ be a finite group and let $\cal C$ be a union of conjugacy
classes of $G$. We denote by $\operatorname{Bk}_{\cal C}(G)$
the set of $\cal C$-blocks of $G$. If $B$ is a $\cal C$-basic
set of $G$ then, for every $b\in\operatorname{Bk}_{\cal C}(G)$,
$B\cap b$ is a $\cal C$-basic set of $b$. Conversely, if every
$b\in\operatorname{Bk}_{\cal C}(G)$ has a $\cal C$-basic set, $B_b$ say,
then $$B=\bigcup_{b\in\operatorname{Bk}_{\cal C}(G)} B_b$$ is a $\cal
C$-basic set of $G$.
\end{lemma}

\begin{proof}
We suppose that $B$ is a $\cal C$-basic set of $G$. Let $b$ be a $\cal
C$-block of $G$. We will prove that $b\cap B$ is a $\cal C$-basic set
of $b$. It is clear that $(b\cap B)^{\cal C}$ is free, being a subset
of $B^{\cal C}$. We have to prove that $(b\cap B)^{\cal C}$ generates
$b^{\cal C}$ over $\Z$. Let $\chi\in b$. Since $B$ is a $\cal C$-basic
set of $G$, there are integers $a_{\chi,\psi}$ ($\psi \in B$) such that
$$\chi^{\cal C}=\sum_{\psi\in B}a_{\chi,\psi}\,\psi^{\cal C}.$$ 
Now take any $\varphi \in \Irr(G)$. If $\varphi\notin b$, then
$\cyc{\chi^{\cal C},\varphi}_G=0$ and $\cyc{\psi^{\cal C},\varphi}_G=0$
for all $\psi \in B\cap b$, so that $\cyc{\chi^{\cal C}-\sum_{\psi\in
B\cap b}a_{\chi,\psi}\,\psi^{\cal C},\varphi}_G=0$. On the other
hand, if $\varphi\in b$, then $\cyc{\psi^{\cal C},\varphi}_G=0$
for all $\psi\notin b$, so that $\cyc{\chi^{\cal C}-\sum_{\psi\in
B\cap b}a_{\chi,\psi}\, \psi^{\cal C},\varphi}_G=\cyc{\chi^{\cal C}-
\sum_{\psi\in B}a_{\chi,\psi}\,\psi^{\cal C},\varphi}_G=0$. Hence we
deduce that
$$\chi^{\cal C}=\sum_{\psi\in B\cap b}a_{\chi,\psi}\,\psi^{\cal C}.$$
Conversely, it is clear that, if $B=\cup_{b\in\operatorname{Bk}_{\cal
C}(G)} B_b$, then $B^{\cal C}$ generates $\Irr(G)^{\cal C}$ over $\Z$.
We now show that $B^{\cal C}$ is free. Suppose there are integers
$a_{\psi}$ ($\psi \in B$) such that
$$\sum_{\psi\in B}a_{\psi}\,\psi^{\cal C}=0.$$ 
Let $b$ be a $\cal C$-block of $G$. The same argument as previously
shows that
$$\sum_{\psi\in B\cap b}a_{\psi}\,\psi^{\cal C}=0.$$ 
Since $B\cap b=B_b$ and $B_b^{\cal C}$ is free, we deduce that
$a_{\psi}=0$ for all $\psi \in b$, and thus for all $\psi \in B$.
\end{proof}

Our main tool is provided by the following

\begin{proposition}\label{BSisometrie}
Let $\cal I_{\eta} \colon (b,\cal C)\rightarrow (b',\cal C')$ be a
generalized perfect isometry between two blocks $b$ and $b'$ of finite
groups $G$ and $H$ respectively. Then $B$ is a $\cal C$-basic set of $b$
if and only if the set $B'={\cal I}(B)=\{\eta(\psi)\cal I_{\eta}(\psi) ,
\; \psi\in B\}$ is a $\cal C'$-basic set of $b'$.
\end{proposition}

\begin{proof}
Take any $\chi\in b$. We have $$
\chi^{\cal
C}=\displaystyle{\sum_{\psi\in\Irr(G)}\cyc{\chi^{\cal
C},\psi}_{G}\,\psi} = \displaystyle{\sum_{\psi\in\Irr(G)}\cyc{\chi,\psi}_{\cal C}\,\psi} = \displaystyle{\sum_{\psi\in b}\cyc{\chi,\psi}_{\cal C}\,\psi}
=\displaystyle{\sum_{\psi\in b}\cyc{\chi^{\cal C},\psi}_{G}\,\psi}.
$$
It follows that
$$\begin{array}{lcl}
\cal I_{\eta}(\chi^{\cal
C})&=&\displaystyle{\cal I_{\eta}\left(\sum_{\psi\in b}\cyc{\chi^{\cal
C},\psi}_{G}\psi\right)}\\
&=&\displaystyle{\sum_{\psi\in b}\cyc{\chi^{\cal
C},\psi}_{G}\,\cal I_{\eta}(
\psi)}\\
&=&\displaystyle{\sum_{\psi\in b}\cyc{\chi
,\psi}_{\cal C}\,\cal I_{\eta}(
\psi)}\\
&=&\displaystyle{\sum_{\psi\in b}\cyc{\cal I_{\eta}(\chi)
,\cal I_{\eta}(\psi)}_{\cal C'}\,\cal I_{\eta}(
\psi)}\\
&=&\displaystyle{\sum_{\psi\in b}\cyc{\cal I_{\eta}(\chi)
,\eta(\psi)\cal I(\psi)}_{\cal C'}\,\eta(\psi)\,\cal I(
\psi)}\\
&=&\displaystyle{\sum_{\psi'\in b'}\cyc{\cal I_{\eta}(\chi)
,\psi'}_{\cal C'}\,
\psi'}\\
&=&\displaystyle{\sum_{\psi'\in b'}\cyc{\cal I_{\eta}(\chi)^{\cal C'}
,\psi'}_{H}\,
\psi'}\\
&=&\displaystyle{\sum_{\psi'\in \Irr(H)}\cyc{\cal I_{\eta}(\chi)^{\cal C'}
,\psi'}_{H}\,
\psi'}\\
&=&\cal I_{\eta}(\chi)^{\cal C'}
\end{array}
$$

We now suppose that $B$ is a $\cal C$-basic set of $b$. Take
any $\theta\in b'$. There exists $\chi\in b$ such that $\cal
I_{\eta}(\chi)=\eta(\chi)\theta$. Since $B$ is a $\cal C$-basic set
of $b$, there are integers $a_{\chi,\psi}$ ($\psi \in B$) such that
$$\chi^{\cal C}=\sum_{\psi\in B}a_{\chi,\psi}\psi^{\cal C}.$$
Using the computation above, we deduce that
$$\eta(\chi)\theta^{\cal C'}={\cal I}_{\eta}(\chi^{\cal C})=\sum_{\psi\in B}a_{\chi,\psi}\cal I_{\eta}(\psi^{\cal C})=\sum_{\psi\in
B}\underbrace{a_{\chi,\psi}\eta(\psi)}_{\in
\Z}\ \underbrace{\cal I(\psi)^{\cal
C'}}_{\in B'^{\cal C'}}.$$
Hence the family $B'^{\cal C'}$ generates $b'^{\cal C'}$ over $\Z$.
We now prove that this family is free. Suppose there are integers
$b_{\psi}$ ($\psi \in B$) such that 
$$\sum_{\psi\in B}b_{\psi} \, \cal I(\psi)^{\cal C'}=\sum_{\psi\in
B}b_{\psi}\,\eta(\psi)\,\cal I_{\eta}(\psi)^{\cal C'}=0.$$
Then we have
$$\cal I_{\eta}\left(\sum_{\psi\in B}b_{\psi}\, \eta(\psi) \, 
\psi^{\cal C}\right)=0.$$ 
Since $\cal I_{\eta}$ is an isomorphism, its kernel is trivial, so that
$$\sum_{\psi\in B}(b_{\psi}\, \eta(\psi) )\, \psi^{\cal C}=0.$$
Now, using the fact that $B$ is a $\cal C$-basic set of $b$, we
deduce that, for all $\psi \in B$, $ b_{\psi} \eta( \psi)=0$, and
thus $b_{\psi}=0$. Hence $B'^{\cal C'}$ is free, and $B'$ is a $\cal
C'$-basic set of $b'$.

\end{proof}

\section{Generalized perfect isometry}\label{part3}

In this section, we present the generalized perfect isometry we are
going to use to reduce our problem from the symmetric group to a wreath
product.

\subsection{Some results on partitions}

Throughout this section, we fix $n \in \mathbb{N}$, $p$ an odd
prime, and $b$ a $p$-block of the symmetric group $\sym_n$. We write
$\Irr(\sym_n)=\{ \chi_{\lambda}, \, \lambda \vdash n \}$ as above.
By the Nakayama Conjecture (cf \cite[6.1.21]{James-Kerber}), the irreducible
complex characters in $b$ are labelled by the partitions of $n$ with
a given $p$-core, $\gamma$ say. In particular, we can consider the
{\emph{weight}} $w$ of $b$, that is the $p$-weight of any partition of
$n$ labelling some irreducible character in $b$. The characters in $b$
can be parametrized by the set of $p${\emph{-tuples of partitions}} of
$w$, using the abacus. Since we will use this description later on,
we present it here on an example. For a complete study, we refer to
\cite[\S2.7]{James-Kerber} (note however that the abacus we describe
here is the horizontal mirror image of that used by James and Kerber).

Consider the lower-right quarter plane (see below). We cut the
axes in segments of length 1, and label these by the integers.
We choose an arbitrary segment on the vertical axis to be the
{\emph{origin}}, indicated by the symbol $\triangleright$. Whenever we
say {\emph{symmetric}} (respectively reflection), we mean symmetric with
respect to (or reflection against) the diagonal $(\Delta)$. We label by
$\triangledown$ the reflection of the origin (this will play a role in
the proof of Lemma \ref{conjugationabacus}).

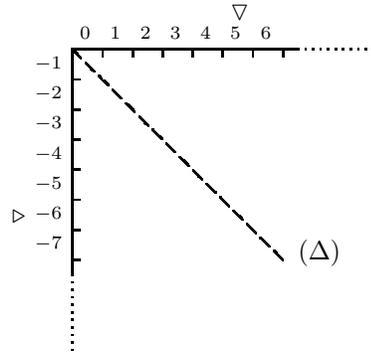
\setlength{\unitlength}{1mm}

\begin{center}

\begin{picture}(44,47)

\linethickness{0.2mm}

\put(0,10) {\line(0,1){30}}

\put(0,40) {\line(1,0){30}}

\dottedline{1}(0,0)(0,10)

\dottedline{1}(30,40)(40,40)

\put(0,12){\line(1,0){1}}

\put(0,16){\line(1,0){1}}

\put(0,20){\line(1,0){1}}

\put(0,24){\line(1,0){1}}

\put(0,28){\line(1,0){1}}

\put(0,32){\line(1,0){1}}

\put(0,36){\line(1,0){1}}

\put(4,39){\line(0,1){1}}

\put(8,39){\line(0,1){1}}

\put(12,39){\line(0,1){1}}

\put(16,39){\line(0,1){1}}

\put(20,39){\line(0,1){1}}

\put(24,39){\line(0,1){1}}

\put(28,39){\line(0,1){1}}

\put(-5,14){$_{-7}$}

\put(-5,18){$_{-6}$}

\put(-8,17){$\triangleright$}

\put(-5,22){$_{-5}$}

\put(-5,26){$_{-4}$}

\put(-5,30){$_{-3}$}

\put(-5,34){$_{-2}$}

\put(-5,38){$_{-1}$}

\put(1, 42){$_{0}$}

\put(5, 42){$_{1}$}

\put(9, 42){$_{2}$}

\put(13, 42){$_{3}$}

\put(17, 42){$_{4}$}

\put(21, 42){$_{5}$}

\put(25, 42){$_{6}$}

\put(21,44){$\triangledown$}

\dashline[+60]{2}(0,40)(28,12)

\put(30,12){$(\Delta)$}

\end{picture}

\end{center}

Now take for example $p=3$ and the partition $\lambda=(4,4,4,3,2)$ of
$n=17$. We put the Young diagram of $\lambda$ in the upper-left corner
of the quarter plane, and consider its (infinite) rim, in bold below.

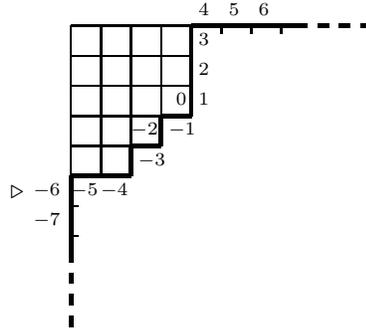
\setlength{\unitlength}{1mm}

\begin{center}

\begin{picture}(41,41)

\linethickness{0.2mm}

\put(0,10) {\line(0,1){30}}

\put(0,40) {\line(1,0){30}}

\put(0,12){\line(1,0){1}}

\put(0,16){\line(1,0){1}}

\put(20,39){\line(0,1){1}}

\put(24,39){\line(0,1){1}}

\put(28,39){\line(0,1){1}}

\put(4,20){\line(0,1){20}}

\put(8,24){\line(0,1){16}}

\put(12,28){\line(0,1){12}}

\put(0,36){\line(1,0){16}}

\put(0,32){\line(1,0){16}}

\put(0,28){\line(1,0){12}}

\put(0,24){\line(1,0){8}}

\linethickness{0.5mm}

\put(0,20) {\line(1,0){8}}

\put(8,24) {\line(1,0){4}}

\put(12,28) {\line(1,0){4}}

\put(16,40) {\line(1,0){14}}

\put(0,10){\line(0,1){10}}

\put(16,28){\line(0,1){12}}

\put(12,24){\line(0,1){4}}

\put(8,20){\line(0,1){4}}

\dashline[+60]{2}(30,40)(40,40)

\dashline[+60]{2}(0,0)(0,10)

\put(-5,14){$_{-7}$}

\put(-5,18){$_{-6}$}

\put(-8,17){$\triangleright$}

\put(0,18){$_{-5}$}

\put(4,18){$_{-4}$}

\put(9,22){$_{-3}$}

\put(8,26){$_{-2}$}

\put(13,26){$_{-1}$}

\put(14, 30){$_{0}$}

\put(17, 30){$_{1}$}

\put(17, 34){$_{2}$}

\put(17, 38){$_{3}$}

\put(17, 42){$_{4}$}

\put(21, 42){$_{5}$}

\put(25, 42){$_{6}$}

\end{picture}

\end{center}

We see the rim as an infinite sequence of vertical and horizontal dashes
of length 1, which, like the axes, we can label by the integers, in such
a way that both labellings coincide whenever they meet. In particular,
whichever partition we take, exactly one dash in the rim is labelled
by the same integer as the origin segment. We say that this dash is
{\emph{labelled by the origin}}, and call it {\emph{origin dash}}. We
can now construct the abacus of $\lambda$ using this sequence of dashes.
We put beads on $p=3$ runners, going from bottom to top and left to
right, putting a bead for each vertical dash and an empty spot for each
horizontal one. For this construction to be uniquely defined, we put the
bead or empty spot corresponding to the origin dash on the first runner.
We get

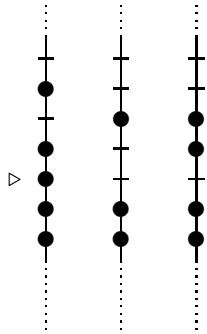
\setlength{\unitlength}{1mm}

\begin{center}

\begin{picture}(24,43)

\put(2,9) {\line(0,1){30}}

\put(12,9) {\line(0,1){30}}

\put(22,9) {\line(0,1){30}}

\put(2,12){\circle*{2}}

\put(12,12){\circle*{2}}

\put(22,12){\circle*{2}}

\put(2,16){\circle*{2}}

\put(12,16){\circle*{2}}

\put(22,16){\circle*{2}}

\put(2,20){\circle*{2}}

\put(11,20){\line(1,0){2}}

\put(2,24){\circle*{2}}

\put(22,24){\circle*{2}}

\put(2,32){\circle*{2}}

\put(11,32){\line(1,0){2}}

\put(1,36){\line(1,0){2}}

\put(1,28) {\line(1,0){2}}

\put(12,28) {\circle*{2}}

\put(22,28) {\circle*{2}}

\put(11,24) {\line(1,0){2}}

\put(21,20) {\line(1,0){2}}

\put(11,36) {\line(1,0){2}}

\put(21,36) {\line(1,0){2}}

\put(21,32) {\line(1,0){2}}

\dottedline{1}(2,38)(2,43)

\dottedline{1}(12,38)(12,43)

\dottedline{1}(22,38)(22,43)

\dottedline{1}(12,0)(12,9)

\dottedline{1}(2,0)(2,9)

\dottedline{1}(22,0)(22,9)

\put(-3,19){$\triangleright$}

\end{picture}

\end{center}

We see that we thus have a bijection which, to a partition, associates
its abacus. If we change the origin, we also change the bijection.
But the bijection only depends on the value modulo $p$ of the integer
labelling the origin segment. Note that we also get these $p$ different
bijections if we fix the origin, but make $p$ different constructions by
changing the runner on which we store the origin dash.

The $p$-information is visible in the abacus in a natural way. For
any positive integer $m$, we call $m${\emph{-hook}} (respectively
$(m)$-{\emph{hook}}) in a partition any hook of length $m$ (respectively
divisible by $m$). If $k$ is a positive integer, then any $kp$-rim-hook
in $\lambda$ corresponds to a bead in the abacus which lies, on the same
runner, $k$ places above an empty spot. Moreover, this correspondence is
bijective.

The removal of a $kp$-rim-hook in the Young diagram is achieved by
moving the corresponding bead down to the empty spot.

By removing all the $p$-hooks from its Young diagram (which is
equivalent to removing all its $p$-rim-hooks), we get the $p$-core
$\gamma(\lambda)$ of $\lambda$, and the corresponding abacus:

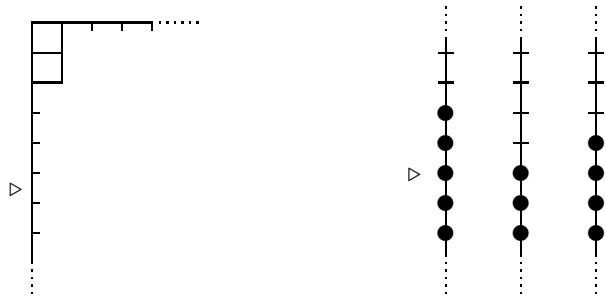
\begin{center}

\begin{picture}(75,42)

\put(0,8) {\line(0,1){32}}

\put(-3,17) {$\triangleright$}

\dottedline{1}(0,8)(0,4)

\put(0,40) {\line(1,0){16}}

\dottedline{1}(16,40)(22,40)

\put(0,36) {\line(1,0){4}}

\put(0,32) {\line(1,0){4}}

\put(4,32){\line(0,1){8}}

\put(0,12) {\line(1,0){1}}

\put(0,16) {\line(1,0){1}}

\put(0,20) {\line(1,0){1}}

\put(0,24) {\line(1,0){1}}

\put(0,28) {\line(1,0){1}}

\put(8,39) {\line(0,1){1}}

\put(12,39) {\line(0,1){1}}

\put(16,39) {\line(0,1){1}}

\put(55,9) {\line(0,1){29}}

\put(65,9) {\line(0,1){29}}

\put(75,9) {\line(0,1){29}}

\put(55,12){\circle*{2}}

\put(65,12){\circle*{2}}

\put(75,12){\circle*{2}}

\put(55,16){\circle*{2}}

\put(65,16){\circle*{2}}

\put(75,16){\circle*{2}}

\put(55,20){\circle*{2}}

\put(65,20){\circle*{2}}

\put(75,20){\circle*{2}}

\put(55,24){\circle*{2}}

\put(64,24){\line(1,0){2}}

\put(54,32){\line(1,0){2}}

\put(55,28){\circle*{2}}

\put(54,36) {\line(1,0){2}}

\put(64,28) {\line(1,0){2}}

\put(74,28) {\line(1,0){2}}

\put(75,24) {\circle*{2}}

\put(74,36) {\line(1,0){2}}

\put(64,36) {\line(1,0){2}}

\put(74,32) {\line(1,0){2}}

\put(64,32) {\line(1,0){2}}

\dottedline{1}(55,38)(55,42)

\dottedline{1}(65,38)(65,42)

\dottedline{1}(75,42)(75,38)

\dottedline{1}(55,4)(55,9)

\dottedline{1}(65,4)(65,9)

\dottedline{1}(75,4)(75,9)

\put(50,19){$\triangleright$}

\end{picture}
\end{center}

To each runner of the abacus, we associate a partition as follows. When
we move all the beads as far down as possible, we get one part for each
bead that we move, of length the number of places the bead goes down.
The resulting $p$-tuple $\alpha_{\lambda}=(\lambda^{(1)}, \, \ldots , \,
\lambda^{(p)})$ of partitions is the $p${\emph{-quotient}} of $\lambda$.
In the above example, we get $\alpha_{\lambda}=(\lambda^{(1)}, \,
\lambda^{(2)} , \, \lambda^{(3)})=((1),(2),(1,1))$. The lengths of the
$\lambda^{(i)}$'s add up to the $p$-weight $w$ of $\lambda$. We write
this as $\alpha_{\lambda} \Vdash w$. For any positive integer $k$,
we call $k$-hook in the quotient of $\lambda$ any $k$-hook in one of
the $\lambda^{(i)}$'s. Then there is a bijection between the set of
$k$-hooks in the $p$-quotient of $\lambda$ and the set of $kp$-hooks in
$\lambda$.

One sees easily that the partition $\lambda$ is uniquely determined by
its $p$-core and $p$-quotient.

\medskip
We now want to study the effect of conjugation on these descriptions.
This is described by the following

\begin{lemma}\label{conjugationabacus}
Conjugation of partitions induces a permutation of order 2 on the
runners of the $p$-abacus, which, up to the choice of the origin,
is the reflection against the middle runner. With such a choice of
origin, if $\lambda$ has $p$-quotient $\alpha_{\lambda}=(\lambda^{(1)},
\, \ldots , \, \lambda^{(p)} )$, then $\lambda^*$ has $p$-quotient
$\alpha_{\lambda^*}=(\lambda^{(p)^*} , \, \ldots , \, \lambda^{(1)^*}
)$.
\end{lemma}
\begin{proof}
Take any partition $\lambda$. Throughout the proof, we use $\lambda$
as above as an example. The rim of $\lambda^*$ is visible in the Young
diagram of $\lambda$: we just need to exchange the roles of horizontal
and vertical dashes in the rim of $\lambda$. In terms of abacus, this
means replacing each bead in the abacus of $\lambda$ by an empty spot,
and each empty spot by a bead. This corresponds to transforming the
partition stored on each runner into its conjugate. If we then turn this
new abacus upside down (that is, rotate it 180 degrees), we obtain {\bf
an} abacus which represents $\lambda^*$. In our example, we get

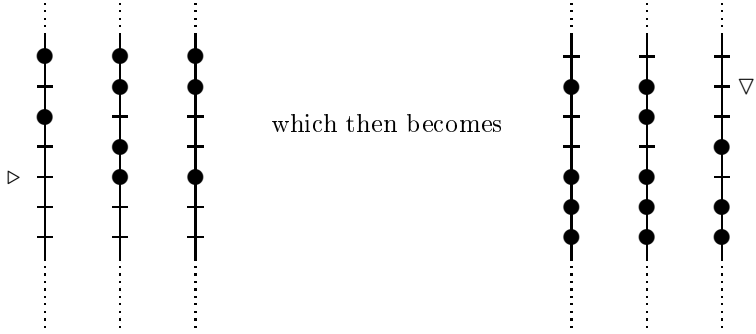
\setlength{\unitlength}{1mm}

\begin{center}

\begin{picture}(100,43)

\put(2,9) {\line(0,1){30}}

\put(12,9) {\line(0,1){30}}

\put(22,9) {\line(0,1){30}}

\put(1,12){\line(1,0){2}}

\put(11,12){\line(1,0){2}}

\put(21,12){\line(1,0){2}}

\put(1,16){\line(1,0){2}}

\put(11,16){\line(1,0){2}}

\put(21,16){\line(1,0){2}}

\put(1,20){\line(1,0){2}}

\put(12,20){\circle*{2}}

\put(1,24){\line(1,0){2}}

\put(21,24){\line(1,0){2}}

\put(1,32){\line(1,0){2}}

\put(12,32){\circle*{2}}

\put(2,36){\circle*{2}}

\put(2,28) {\circle*{2}}

\put(11,28) {\line(1,0){2}}

\put(21,28) {\line(1,0){2}}

\put(12,24) {\circle*{2}}

\put(22,20) {\circle*{2}}

\put(12,36) {\circle*{2}}

\put(22,36) {\circle*{2}}

\put(22,32) {\circle*{2}}

\dottedline{1}(2,38)(2,43)

\dottedline{1}(12,38)(12,43)

\dottedline{1}(22,38)(22,43)

\dottedline{1}(12,0)(12,9)

\dottedline{1}(2,0)(2,9)

\dottedline{1}(22,0)(22,9)

\put(-3,19){$\triangleright$}

\put(32,26){$\mbox{which then becomes}$}

\put(72,9) {\line(0,1){30}}

\put(82,9) {\line(0,1){30}}

\put(92,9) {\line(0,1){30}}

\put(72,12){\circle*{2}}

\put(82,12){\circle*{2}}

\put(92,12){\circle*{2}}

\put(72,16){\circle*{2}}

\put(82,16){\circle*{2}}

\put(92,16){\circle*{2}}

\put(72,20){\circle*{2}}

\put(82,20){\circle*{2}}

\put(91,20) {\line(1,0){2}}

\put(71,24){\line(1,0){2}}

\put(81,24) {\line(1,0){2}}

\put(92,24){\circle*{2}}

\put(71,28) {\line(1,0){2}}

\put(82,28) {\circle*{2}}

\put(91,28) {\line(1,0){2}}

\put(72,32){\circle*{2}}

\put(82,32){\circle*{2}}

\put(91,32) {\line(1,0){2}}

\put(71,36){\line(1,0){2}}

\put(81,36) {\line(1,0){2}}

\put(91,36) {\line(1,0){2}}

\dottedline{1}(72,38)(72,43)

\dottedline{1}(82,38)(82,43)

\dottedline{1}(92,38)(92,43)

\dottedline{1}(82,0)(82,9)

\dottedline{1}(72,0)(72,9)

\dottedline{1}(92,0)(92,9)

\put(94,31){$\triangledown$}

\end{picture}

\end{center}

However, this is {\bf the} abacus of $\lambda^*$ (in the bijection
we mentionned before) only if the position labelled by the origin
$\triangleright$ is on the first runner. The spot that was labelled by
$\triangleright$ in the abacus of $\lambda$ is now on the last runner,
and labelled by the symmetrized $\triangledown$ of $\triangleright$ on
the horizontal axis. Hence the spot now labelled by $\triangleright$
is on the first runner if and only if the number of spots which
separate it from the spot labelled by $\triangledown$ is $p-1$ plus a
multiple of $p$. But this number is exactly the distance, following
the axes, between $\triangledown$ and $\triangleright$. Hence our
construction gives {\bf the} abacus of $\lambda^*$ if and only if
this distance is congruent to $-1$ modulo $p$ (as is the case in
our example, since $11 = -1 \; (mod \; 3)$). If this is the case,
then, from the description we gave above, we see immediately that,
if $\lambda$ has $p$-quotient $\alpha_{\lambda}=(\lambda^{(1)}, \,
\ldots , \, \lambda^{(p)} )$, then $\lambda^*$ has $p$-quotient
$\alpha_{\lambda^*}=(\lambda^{(p)^*} , \, \ldots , \, \lambda^{(1)^*} )$
(in our example, we get that $\alpha_{\lambda^*}=(\lambda^{(3)^*}, \,
\lambda^{(2)^*} , \, \lambda^{(1)^*})=((2),(1,1),(1))$).
\end{proof}

This leads to the following

\begin{convention}\label{convention}
We choose the origin so that the permutation of Lemma
\ref{conjugationabacus} is the reflection against the $r$-th runner,
where $r = (p+1)/2$.
\end{convention}

\begin{remark}
Note that, with Convention \ref{convention}, applying the transformation
described in the proof of Lemma \ref{conjugationabacus} to the abacus
of $\lambda$ always gives the abacus of $\lambda^*$. However, it
corresponds to $(\gamma(\lambda), \, \alpha_{\lambda}) \longmapsto
(\gamma(\lambda)^*, \, \alpha_{\lambda^*})$, and it's easy to see
that $\gamma(\lambda)^* = \gamma(\lambda^*)$. In particular, this
transformation needs not preserve the core $\gamma(\lambda)$, and
thus the block in which lies the character $\chi_{\lambda}$. If
$\gamma(\lambda) \neq \gamma(\lambda^*)$, then there still is a
character $\chi_{\mu}$ in the same $p$-block as $\chi_{\lambda}$ and
such that $\alpha_{\mu}=\alpha_{\lambda^*}$, but $\mu \neq \lambda^*$.
\end{remark}

We also include here a lemma about the partition $\overline{\lambda}$
defined in Section \ref{intro} for a self-conjugate partition, and which
will be useful later.

\begin{lemma}\label{preg}
Let $p$ be an odd prime, $r = (p+1)/2$, and let $\lambda$ be a
self-conjugate partition of $n$. Then the partition $\overline{\lambda}$
is $p$-regular, if and only if the $r$-th part of the quotient
$\alpha_{\lambda}$ of $\lambda$ is empty.
\end{lemma}

\begin{proof}
A hook in the Young diagram of a partition is called {\emph{diagonal}}
if its top left box lies on the diagonal $(\Delta)$. For any
self-conjugate $\lambda \vdash n$, the parts of $\overline{\lambda}$ are
therefore exactly the lengths of the diagonal hooks in $\lambda$.

Now, by conjugation, any $(p)$-hook stored on the $i$-th runner of the
abacus of (any) $\lambda$ is tranformed into its reflection, and is
stored on the $(p+1-i)$-th runner in the abacus of $\lambda^*$. Thus the
hooks on the $i$-th runner of $\lambda$ are exactly the reflections of
those on the $(p+1-i)$-th runner of $\lambda^*$.

In particular, if $\lambda$ is self-conjugate, then
$\alpha_{\lambda}=\alpha_{\lambda^*}$, and each of the hooks on the
$r$-th runner is its own reflection (and all the others move). Hence the
$r$-th runner in the abacus of $\lambda$ stores exactly the symmetric
$(p)$-hooks, which are precisely the diagonal $(p)$-hooks. Since these
correspond to the parts divisible by $p$ in $\overline{\lambda}$, we
easily get the result.

\smallskip
\noindent
Remark: more generally, even if $\lambda$ is not self-conjugate, one can
prove that the $r$-th runner in the abacus of $\lambda$ stores exactly
the diagonal $(p)$-hooks of $\lambda$.

\end{proof}

\begin{remark}\label{remw=0}
In particular, if $w=0$, then $\lambda$ has no $(p)$-hook (and, a
fortiori, no diagonal $(p)$-hook), so that $\overline{\lambda}$ is
always $p$-regular.
\end{remark}

\subsection{Wreath product}\label{wp}

We denote by $\mathbb{Z}_p$ and $\mathbb{Z}_{p-1}$ the cyclic groups
of order $p$ and $p-1$ respectively. Note that $\mathbb{Z}_{p-1} \cong
\Aut(\mathbb{Z}_p)$, so that we can construct the semidirect product $N
= \mathbb{Z}_p \rtimes \mathbb{Z}_{p-1} $. Considering $\mathbb{Z}_p$
as the subgroup of $\sym_p$ generated by a $p$-cycle, we then have $N
= N_{\sym_p}(\mathbb{Z}_p)$. We denote by $G_{p,w}$ the wreath product
$N \wr \sym_w$. That is, $G_{p,w}$ is the semidirect product $N^w
\rtimes \sym_w$, where $\sym_w$ acts by permutation on the $w$ copies
of $N$. Note that, if $w <p$, then $P \cong \mathbb{Z}_p^w$ is a Sylow
$p$-subgroup of $\sym_{pw}$, and $G_{p,w} \cong N_{\sym_{pw}}(P)$.

For a complete description of wreath products and their representations,
we refer to \cite[Chapter 4]{James-Kerber}.

We write $r = (p+1)/2$. We have $\Irr(N)=\{ \psi_1, \, \ldots , \,
\psi_p \}$, with $\psi_i(1)=1$ for $1 \leq i \leq p$ and $i \neq r$,
and $\psi_r(1)=p-1$. More precisely, writing $1_{\mathbb{Z}_p}$ for the
trivial character of $\mathbb{Z}_p$, we have
$$
\Ind_{\mathbb{Z}_p}^N(1_{\mathbb{Z}_p})= \displaystyle \sum_{
\substack{ i=1 \\ i \neq r}}^p \psi_i \; \; \mbox{and} \; \;
\Res_{\mathbb{Z}_p}^N(\psi_i)= 1_{\mathbb{Z}_p} \; \; (1 \leq i \leq p,
\, i\neq r),
$$
and
$$
\Res_{\mathbb{Z}_p}^N(\psi_r)= \eta_2 + \, \cdots \, + \eta_p \; \;
\mbox{and} \; \; \Ind_{\mathbb{Z}_p}^N(\eta_i)= \psi_r \; \; (2 \leq i
\leq p),
$$
where $\{ \eta_2 , \, \ldots , \, \eta_p \} = \Irr (\mathbb{Z}_p)
\setminus \{ 1_{\mathbb{Z}_p} \}$.

\smallskip
The irreducible complex characters of $G_{p,w}=N \wr \sym_w$ are
parametrized by the $p$-tuples of partitions of $w$ as follows. Take
any $\alpha=(\alpha^1, \, \ldots , \, \alpha^p) \Vdash w$ and consider
the irreducible character $\prod_{i=1}^p \psi_i^{|\alpha^i|}$ of the
{\emph{base group}} $N^w$. It can be extended in a natural way to its
inertia subgroup $N \wr \sym_{|\alpha^1|} \times \cdots \times N \wr
\sym_{|\alpha^p|}$, giving the irreducible character $\prod_{i=1}^p
\widetilde{\psi_i^{|\alpha^i|}}$ \cite[p.\,154]{James-Kerber}. Any
extension is of the form $\prod_{i=1}^p (\widetilde{\psi_i^{|\alpha^i|}}
\otimes \varphi_{\alpha_i})$, with $\varphi_{\alpha_i} \in
\Irr(\sym_{|\alpha^i|})$ ($1 \leq i \leq p$). Then $\chi^{\alpha}:=
\Ind_{\scriptstyle \prod_{i=1}^p N \wr \sym_{|\alpha^i|} }^{N \wr
\sym_w} ( \prod_{i=1}^p \widetilde{\psi_i^{|\alpha^i|}} \otimes
\varphi_{\alpha^i} ) \in \Irr(G_{p,w})$. Different $\alpha \Vdash w$
give different irreducible characters of $G_{p,w}$, and any irreducible
character of $G_{p,w}$ can be obtained in this way.

\smallskip
Note for future reference that, in the above notation, if
$\alpha^r = \emptyset$, then $P=\mathbb{Z}_p^w \subseteq \ker
(\chi^{\alpha})$. Indeed, for all $1 \leq i \leq p$, $i \neq r$,
we have $ \Res_{\mathbb{Z}_p}^N(\psi_i)= 1_{\mathbb{Z}_p}$; thus
$\Res_{\mathbb{Z}_p^{|\alpha^i|}}^{N^{|\alpha^i|}}(\psi_i^{|\alpha^i|})=
1_{\mathbb{Z}_p^{|\alpha^i|}}$, so that
$\widetilde{\psi_i^{|\alpha^i|}}(g)=\widetilde{\psi_i^{|\alpha^i|}}(1)$
for all $g \in \mathbb{Z}_p^{|\alpha^i|}$,
and $(\widetilde{\psi_i^{|\alpha^i|}} \otimes
\varphi_{\alpha_i})(g)=(\widetilde{\psi_i^{|\alpha^i|}} \otimes
\varphi_{\alpha_i})(1)$ for all $g \in \mathbb{Z}_p^{|\alpha^i|}$. Since
$P \leq N^w \triangleleft N \wr \sym_w$, we easily get that, for all $g
\in P=\mathbb{Z}_p^w$, $\chi^{\alpha}(g)=\chi^{\alpha}(1)$.

\medskip

The conjugacy classes of $G_{p,w}$ are also parametrized by the
$p$-tuples of partitions of $w$. Let $g_1, \, \ldots , \, g_p$ be
representatives for the conjugacy classes of $N$. Note that, when
seen as a subgroup of $\sym_p$, $N$ has a unique conjugacy class of
$p$-cycles, for which we take representative $g_p$. The elements of $G_{p,w}=N
\wr \sym_w$ are of the form $(h, \, \gamma)=((h_1, \, \ldots, \, h_w ) , \,
\sigma)$, with $h_1, \, \ldots, \, h_w \in N$ and $ \sigma \in \sym_w$.
For any such element, and for any $k$-cycle $\kappa=(j, j\kappa, \ldots,
j \kappa^{k-1})$ in $\sigma$, we define the {\emph{cycle product}} of
$(h, \, \sigma)$ and $\kappa$ by
$$
g((h, \, \sigma); \, \kappa)= h_j h_{j \kappa^{-1}} h_{j \kappa^{-2}}
\ldots h_{j \kappa^{-(k-1)}}.
$$
In particular, $g((h, \, \sigma) ; \, \kappa) \in N$. If $\sigma$ has
cycle structure $\pi$ say, then we form $p$ partitions $(\pi_1, \,
\ldots, \, \pi_p)$ from $\pi$ as follows: any cycle $\kappa$ in $\pi$
gives a cycle of the same length in $\pi_i$ if the cycle product $g((h ,
\, \sigma); \, \kappa)$ is conjugate to $g_i$ in $N$. The resulting $p$-tuple
of partitions of $w$ describes the {\emph{cycle structure}} of $(h, \,
\gamma)$. Two elements of $G_{p,w}$ are conjugate if and only if they
have the same cycle structure.


\smallskip
Define ${\cal C}_{\emptyset}$ to be the set of elements of
$G_{p,w}$ with cycle structure $(\pi_1, \, \ldots, \, \pi_p)$ such that
$\pi_p = \emptyset$.

Note that, in the natural embedding of $G_{p,w}$ in $\sym_{p w}$, for
any $k$-cycle $\kappa$ of $\sigma$, each $m$-cycle of the cycle product
$g((h, \, \sigma) ; \, \kappa) \in \sym_p$ corresponds to $gcd(m, \, k)$ cycles of
length $lcm(m, \, k)$ in $(h, \, \sigma) \sym_{pw}$. In particular, if $w < p$,
then ${\cal C}_{\emptyset}$ is the set of $p$-regular elements of $G_{p,w}$.

\subsection{A generalized perfect isometry}

We now come to the generalized perfect isometry presented by the
second auther in \cite{JB}. We first introduce one last piece of
notation (\cite[Prop. 3.8]{JB}). Let $\alpha=(\alpha^1, \,
\ldots , \, \alpha^p) $ be any $p$-tuple of partitions of $w$, and
$r = (p+1)/2$ as before. We define $\tilde{\alpha} \Vdash w$ by
$$\tilde{\alpha}=(\alpha^1, \, \ldots , \, \alpha^{r-1}, \, (\alpha^r)^*,
\, \alpha^{r+1}, \, \ldots , \, \alpha^p),$$ where $^*$ denotes
conjugation of partitions.

Note that $\tilde{ \; }$ is a bijection from $\{ (\alpha^1, \, \ldots
, \, \alpha^p) \Vdash w \}$ onto itself, and is the identity on
$\{ (\alpha^1, \, \ldots , \, \alpha^p) \Vdash w , \, \alpha^r =
\emptyset\}$.

\smallskip
The main result in \cite{JB} can, in our context, be stated as follows

\begin{theorem}{( \cite{JB}, Theorem 4.1)}\label{perfisom}
Let $n \geq 1$ be an integer and $p$ be a prime, and let $b$ be a
$p$-block of weight $w \neq 0$. For any $\chi_{\lambda} \in b$, denote
by $\alpha_{\lambda}$ the $p$-quotient of $\lambda$. Then, with the
above notations, the bijection
$$
\begin{array}{rcl} {\cal J} \colon b & \longrightarrow
& \Irr(G_{p,w}) \\ \chi_{\lambda} & \longmapsto &
\chi^{\tilde{\alpha_{\lambda}}} \end{array}
$$
induces a generalized perfect isometry ${\cal J}_{\eta} \colon (b,
p\textrm{-reg}) \rightarrow (\Irr(G_{p,w}), {\cal C}_{\emptyset})$.
\end{theorem}

\begin{proof}
We give here the main lines of the proof, first for the comfort of
the reader, but also because, in fact, Theorem 4.1 in \cite{JB} is
not stated as above; it is only stated in the case where $b$ is the
principal $p$-block of $\sym_n$, and when $w<p$. However, for our
purpose, and for the version of this theorem we give, these two
hypotheses are not necessary, as we will try and explain below. The
object of \cite{JB} is to give an analogue, for the $\ell$-blocks
of $\sym_n$ (cf \cite{KOR}), where $\ell$ is an arbitrary integer,
of Brou\'e's Abelian Defect Conjecture (hence the hypotheses on $b$
and $w$). In particular, in this case, and when $\ell=p$, ${\cal
C}_{\emptyset}$ is the set of $p$-regular elements of $G_{p,w}$.
This fails when $w \geq p$, but the generalized perfect isometry we
mentionned above still exists.

\medskip
We consider a $p$-cycle $\omega$ in $\sym_p$ and $L = \cyc{ \omega}$,
and we let $N=N_{\sym_p}(L)$. Then $N = L \rtimes \Aut (L) \cong \Z_p
\rtimes \Z_{p-1}$, and direct computation shows that $N$ has a single
$p$-block.

Now take any $p$-block $b$ of weight $w \neq 0$ of $\sym_n$. Take
representatives $(h_i= \omega^{i-1}, \, 1 \leq i \leq p)$ for the
conjugacy classes of $L$. As above, we have that the conjugacy classes
of $L \wr \sym_w$ are parametrized by the $p$-tuples $(\pi_1, \, \ldots
, \, \pi_p) \Vdash w$ of partitions of $w$. For $\pi_1$ in $\cal P_{\leq
w}$ (the set of partitions of at most $w$), we write $\cal D_{\pi_1}$
for the set of elements whose cycle structure has first part $\pi_1$. In
particular, $\cal D_{\emptyset}$ is the set of {\emph{regular}} elements
described by Osima (cf \cite{Osima}). We then have
$$L \wr \sym_w = \displaystyle \coprod_{\pi_{1} \in {\cal P}_{\leq
w}} {\cal D}_{\pi_{1}}.$$
Write $\Irr(L \wr \sym_w)=\{ \xi^{\alpha}, \, \alpha \Vdash w \}$.
One fundamental result we use (cf \cite{Osima} and \cite{KOR}), and
which does not depend on $w\neq 0$, is that, with the notations above,
$\lambda \longmapsto \alpha_{\lambda}$ induces a generalized perfect
isometry ${\cal I}_{\epsilon} \colon (b, p\textrm{-reg}) \rightarrow
(\Irr(L \wr \sym_w), {\cal D}_{\emptyset})$. We therefore have, for all
$\chi_{\lambda}, \, \chi_{\mu} \in b$,
\begin{equation}\label{perfisom1} \cyc{\chi_{\lambda}, \, \chi_{\mu}}_{p\textrm{-reg}} =  \cyc{ \epsilon(\lambda) \xi^{\alpha_{\lambda}}, \, \epsilon(\mu) \xi^{\alpha_{\mu}}}_{{\cal D}_{\emptyset}}. 
\end{equation}
The biggest part of the proof now consists in transporting this to the
group $G_{p,w}$ the elements of ${\cal C}_{\emptyset}$. The idea is to
work with {\emph{singular}} elements instead of regular elements, and
use induction on $w$. For any $\alpha, \, \beta \Vdash w$, we have
$$\langle \xi^{\alpha}, \,
\xi^{\beta}\rangle _{\scriptstyle \cal D_{\emptyset}}=\delta_{\alpha \beta} - \langle \xi^{\alpha}, \,
\xi^{\beta}\rangle _{ \scriptstyle L \wr \sym_w \setminus \cal D_{\emptyset}}$$and
$$\langle \xi^{\alpha}, \,
\xi^{\beta}\rangle _{\scriptstyle L \wr \sym_w \setminus \cal D_{\emptyset}}=
\displaystyle \sum_{ \emptyset \neq \pi_{1} \in {\cal P}_{\leq
w}} \langle \xi^{\alpha}, \, \xi^{\beta}\rangle _{\scriptstyle
{\cal D}_{\pi_1}}.$$
The induction on $w$ will be based on an analogue of the
Murnaghan-Nakayama Rule which holds in wreath products (cf
\cite[4.4]{Pfeiffer}): writing, for any $\emptyset \neq \pi_{1}
\in {\cal P}_{\leq w}$, ${\cal D}_{(\pi_1, \ldots, \pi_{p})}$ for the
conjugacy class of cycle structure $(\pi_1, \ldots, \pi_{\ell})$ of $L
\wr \sym_w$, and ${\cal D}_{(\pi_2, \ldots, \pi_{p})}^{w- |\pi_{1}|}$
for the set of elements of cycle structure $(\emptyset, \pi_2, \ldots,
\pi_{p})$ in $L \wr \sym_{w-|\pi_1|}$, we have
$$\xi^{\alpha}({\cal D}_{(\pi_1, \ldots, \pi_{p})} )= \displaystyle
\sum_{\scriptstyle 1 \leq s_1, \ldots,s_i \leq p} \; \;
\sum_{\scriptstyle \hat{\alpha} \in {\cal L}_{\alpha,(s_1,
\ldots,s_i)}^{\pi_{1}}} (-1)^{L_{\alpha \hat{\alpha}}}
\xi^{\hat{\alpha}}({\cal D}_{(\pi_2, \ldots,
\pi_{p})}^{w-|\pi_1|} ),$$
where ${\cal L}_{\alpha,(s_1, \ldots,s_i)}^{\pi_1}$ is the set of
$p$-tuples of partitions of $w-|\pi_1|$ which can be obtained from
$\alpha$ by removing successively a $k_1$-hook from $\alpha^{s_1}$,
then a $k_2$-hook from the ``$s_2$-th coordinate'' of the resulting
$p$-tuple of partitions of $w-k_1$, etc, and finally a $k_i$-hook from
the ``$s_i$-th coordinate'' of the resulting $p$-tuple of partitions
of $w-(k_1+ \cdots +k_{i-1})$, and, for $ \hat{\alpha} \in {\cal
L}_{\alpha,(s_1, \ldots,s_i)}^{\pi_1}$, $L_{\alpha \hat{\alpha}}$ is
the sum of the leg-lengths of the hooks removed to get from $\alpha$ to
$\hat{\alpha}$.

\smallskip
Writing $b_k$ for the number of $k$-cycles in $\pi_1$, and  $d_{\pi_1}=\prod_{\scriptstyle 1 \leq k \leq w} b_{ k}! (k p)^{b_{ k}}$, we then get, writing ${\bf s}=(s_1, \, \ldots , \, s_i)$,
$$\langle \xi^{\alpha}, \, \xi^{\beta}\rangle _{\scriptstyle
{\cal D}_{\pi_1}}= \displaystyle \frac{1}{d_{\pi_1}}
\sum_{\scriptstyle 1 \leq {\bf s} \leq p \atop 1 \leq {\bf t}
\leq p} \sum_{\scriptstyle \hat{\alpha} \in {\cal
L}_{\alpha,{\bf s}}^{\pi_1} \atop \hat{\beta} \in {\cal
L}_{\beta,{\bf t}}^{\pi_1} } (-1)^{L_{\alpha
\hat{\alpha}}}(-1)^{L_{\beta \hat{\beta}}} \langle
\xi^{\hat{\alpha}}, \, \xi^{\hat{\beta}} \rangle_{ {\cal
D}_{\emptyset}^{w- |\pi_1|}}.$$ We will use this formula to
exhibit a generalized perfect isometry between $\Irr(L \wr \sym_w)$
and $\Irr(N \wr \sym_w)$.

\smallskip
Take $(g_1=1, \, \ldots , \, g_p=\omega)$ representatives for the conjugacy classes of $N$. We have 
$$N \wr \sym_w = \displaystyle \coprod_{\pi_{p} \in {\cal P}_{\leq
w}} {\cal C}_{\pi_{p}},$$
where $\cal C_{\pi_p}$ denotes the set of elements whose cycle structure has $p$-th part $\pi_p$, and, for any $\alpha, \, \beta \Vdash w$,
$$\langle \chi^{\alpha}, \,
\chi^{\beta}\rangle _{\scriptstyle N \wr \sym_w \setminus \cal C_{\emptyset}}=
\displaystyle \sum_{ \emptyset \neq \pi_{p} \in {\cal P}_{\leq
w}} \langle \chi^{\alpha}, \, \chi^{\beta}\rangle _{\scriptstyle
{\cal C}_{\pi_p}}.$$
Now fix any $\emptyset \neq \pi_{p} \in {\cal P}_{\leq
w}$, and write $b'_k$ for the number of $k$-cycles in $\pi_p$, and  $c_{\pi_p}=\prod_{\scriptstyle 1 \leq k \leq w} b'_{ k}! (k p)^{b'_{ k}}$. In particular, if $\pi_1=\pi_p$, then $d_{\pi_1}=c_{\pi_p}$. We then have, by repeated use of the Murnaghan-Nakayama Rule, and with similar notations as above, 
$$\langle \chi^{\alpha}, \, \chi^{\beta}\rangle _{\scriptstyle
{\cal C}_{\pi_p}}= \displaystyle \frac{1}{c_{\pi_p}}
\sum_{\scriptstyle 1 \leq {\bf s} \leq p \atop 1 \leq {\bf t} \leq
p}  \; \sum_{\scriptstyle \hat{\alpha} \in {\cal L}_{\alpha,{\bf
s}}^{\pi_p}
 \atop \hat{\beta} \in {\cal L}_{\beta,{\bf s}}^{\pi_p} }  \Psi_{{\bf s}}(g_p)
   \overline{\Psi_{{\bf t}}(g_p)} (-1)^{L_{\alpha \hat{\alpha}}}(-1)^{L_{\beta \hat{\beta}}}
    \langle \chi^{\hat{\alpha}}, \, \chi^{\hat{\beta}} \rangle_{ {\cal C}_{\emptyset}^{w- |\pi_p|}},$$
where
$\Psi_{{\bf s}}(g_p)=\psi_{s_1}(g_p) \ldots \psi_{s_i}(g_p)$ (with $\Irr(N)= \{ \psi_1, \, \ldots , \, \psi_p \}$).
Now the transformation $\hat{ \; }$ defined before Theorem \ref{perfisom} is introduced precisely to get rid of the factor $\Psi_{{\bf s}}(g_p)$ (which is only problematic when ${\bf s}$ contains $r=(p+1)/2$, since $\psi_r(\omega)=-1$ but $\psi_i(\omega)=1$ for $i \neq r$). For $\alpha \Vdash w$, we let $\chi_0^{\alpha}=(-1)^{|\alpha_r|}\chi^{\hat{\alpha}}$, and we can use induction on $w$ to prove that, for any $\alpha, \, \beta \Vdash w$,
\begin{equation}\label{perfisom2}
\cyc{\chi_0^{\alpha}, \, \chi_0^{\beta}}_{\cal C_{\emptyset}} =  \cyc{
\xi^{\alpha}, \,  \xi^{\beta}}_{{\cal D}_{\emptyset}} 
\end{equation}
(cf \cite[3.10]{JB}). Note that the result is clearly true when
$w=0$ since $\cal C_{\emptyset}=\cal D_{\emptyset}=\emptyset$ and
$\chi^{\emptyset}=\xi^{\emptyset}=1$. Composing the results of
(\ref{perfisom1}) and (\ref{perfisom2}), we get the desired generalized
perfect isometry.
\end{proof}

Note that the above result does not depend on the order in which we list
the irreducible characters of $N$ (provided we modify the definition of
$\tilde{ \; }$ in accordance), or on the runner of the abacus we choose
to be fixed by conjugation. Each choice of labelling for the characters
of $N$ and for the fixed runner will give an (a priori different)
generalized perfect isometry. In particular, it was legitimate to make
the choices we made.

\section{A ${\cal C}_{\emptyset}$-basic set for $G_{p,w}$}\label{part4}

We now want to exhibit a ${\cal C}_{\emptyset}$-basic set for $G_{p,w}$,
which we will then transport to $\sym_n$ using the generalized perfect
isometry given by Theorem \ref{perfisom}. Our main tool to do so will be
Proposition \ref{propsemidirect}. In the course of its proof, we will
need the following general lemma about semidirect products, which we
therefore state separately.

\begin{lemma}\label{lemsemidirect}
Let $G = H \rtimes K$ be a finite group, and $\pi \colon G
\longrightarrow K$ be the canonical surjection. If $k \in K$ is
conjugate in $G$ to $g \in G$, then $k$ is conjugate in $K$ to $\pi(g)$.
In particular, two elements of $K$ are conjugate in $G$ if and only if
they are conjugate in $K$.
\end{lemma}

\begin{proof}
Suppose $g \in G$ and $k \in K$ are conjugate in $G$. Then there exist
$h \in H$ and $\ell \in K$ such that $g=(h\ell)k(h\ell)^{-1}$. Thus
$g = h (\ell k \ell^{-1}) h^{-1}$, and $\ell k \ell^{-1}=k' \in K$.
Hence $g = (h (h^{-1})^{k'})k'$, and, by uniqueness of the $(H, \,
K)$-decomposition in $G$, we must have $k'= \pi(g)$, so that $k$ and
$\pi(g)$ are conjugate in $K$.
\end{proof}

\begin{proposition}\label{propsemidirect}
Suppose $G = H \rtimes K$ is a finite group, and $\pi \colon G
\longrightarrow K$ is the canonical surjection. Suppose ${\cal C}$ is
a union of conjugacy classes of $G$ such that $K \subseteq {\cal C}$
and $|Cl_G({\cal C})| = |Cl_K(K)|$ (i.e. the numbers of $G$-conjugacy
classes in ${\cal C}$ and of conjugacy classes in $K$ are the same).
Then $\{ \psi \circ \pi, \, \psi \in \Irr(K) \}$, the set of irreducible
characters of $G$ with $H$ in their kernel, is a ${\cal C}$-basic set
for $G$.
\end{proposition}
Note that Proposition \ref{propsemidirect} is known in the case where
$H$ is a $p$-group, $K$ is a $p'$-group, and ${\cal C}$ is the set of
$p$-regular elements of $G$.
\begin{proof}

Let ${\cal B}=\{ \psi \circ \pi, \, \psi \in \Irr(K) \}$. Since
$|Cl_G({\cal C})| = |Cl_K(K)|=|{\cal B}|$, it suffices to prove that
${\cal B}^{\cal C}$ generates $\Irr(G)^{\cal C}$ over $\Z$.

\smallskip
Since ${\cal K} \subseteq {\cal C}$, and since, by Lemma
\ref{lemsemidirect}, any two elements of $K$ are conjugate in $K$ if and
only if they are conjugate in $G$, we get that $|Cl_K(K)| = | \{ Cl_G(g)
\in Cl_G({\cal C}) \; \mbox{such that} \; \exists g_K \in K, \, Cl_G(g)
= Cl_G(g_K) \} |$. Since $|Cl_G({\cal C})| = |Cl_K(K)|$, this shows that
every $G$-conjugacy class in ${\cal C}$ has a representative in $K$.

\smallskip
Now take any $\chi \in \Irr(G)$, and write $\Res_K^G(\chi) = \sum_{\psi
\in \Irr(K)} a_{\chi \psi} \psi$, with
$a_{\chi \psi} \in \N$ for all $\psi \in \Irr(K)$. Take any $g \in
{\cal C}$. By the above, there exists $g_K \in K$ such that $Cl_G(g) =
Cl_G(g_K)$, and, by Lemma \ref{lemsemidirect}, $g_K$ and $\pi(g)$ are
conjugate in $K$. We therefore get $$\chi(g)= \chi(g_K) = \displaystyle
\sum_{\psi \in \Irr(K)} a_{\chi \psi} \psi (g_K) = \sum_{\psi \in
\Irr(K)} a_{\chi \psi} \psi( \pi(g) )$$ (the first equality being true
because $g$ and $g_K$ are conjugate in $G$, the last because $g_K$ and
$\pi(g)$ are conjugate in $K$). Since this holds for any $g \in {\cal
C}$, we finally get that $\chi^{\cal C}= \sum_{\psi \in \Irr(K)} a_{\chi
\psi} (\psi \circ \pi)^{\cal C}$, so that ${\cal B}^{\cal C}$ generates
$\Irr(G)^{\cal C}$ over $\Z$.

\end{proof}

We can now prove the main result of this section

\begin{theorem}\label{basicsetwreath}
The set $B_{\emptyset}=\{ \chi^{\alpha} \in \Irr(G_{p,w}) \, | \,
\alpha=(\alpha^1, \, \ldots , \, \alpha^p) \Vdash w, \, \alpha^r=
\emptyset \}$ is a ${\cal C}_{\emptyset}$-basic set for $G_{p,w}$.
\end{theorem}

\begin{proof}

We start by proving that we can somehow apply Proposition
\ref{propsemidirect} to $G_{p,w}$ to obtain a ${\cal
C}_{\emptyset}$-basic set. Write $G_{p,w}=(\Z_p \rtimes \Z_{p-1} ) \wr \sym_w$
as
$$
\{ ( ( (a_1, \, \ldots , \, a_w) , \, (b_1, \, \ldots , \, b_w) ), \,
\sigma ) , \; a_i \in \Z_p, \, b_i \in \Z_{p-1} \, (1 \leq i \leq w), \,
\sigma \in \sym_w \}.
$$
One proves easily, using the properties of semidirect products and
the fact that $\sym_w$ acts by permutation on the $w$ copies of $\Z_p
\rtimes \Z_{p-1} $, the following:

\begin{itemize}

\item{}
${\cal H}= \{ ( ( (a_1, \, \ldots , \, a_w) , \, (1, \, \ldots , \, 1)
), \, 1 ) \}$ is a normal subgroup of $G_{p,w}$, isomorphic to $\Z_p^w$,

\item{}
  $\{ ( ( (1, \, \ldots , \, 1) , \, (1, \, \ldots , \, 1) ), \, \sigma )
\}$ is a subgroup of $G_{p,w}$, isomorphic to $\sym_w$,

  \item
  ${\cal K}= \{ ( ( (1, \, \ldots , \, 1) , \, (b_1, \, \ldots , \, b_w)
), \, \sigma ) \}$ is a subgroup of $G_{p,w}$, isomorphic to $ \Z_{p-1} \wr
\sym_w$,

  \item
  ${\cal H} \cap {\cal K} = \{1\}$, and ${\cal H} {\cal K}=G_{p,w}$.

\end{itemize}
We can therefore write $G_{p,w}$ as a semidirect product ${\cal H} \rtimes
{\cal K}$.


For any $k = ( ( (1, \, \ldots , \, 1) , \, (b_1, \, \ldots , \, b_w)
), \, \sigma ) \in {\cal K}$, we see that all cycle products of $k$
are products of some $b_i$'s, hence belong to $\Z_{p-1}$, and, in
particular, are not conjugate to the $p$-cycle of $N=\Z_p \rtimes
\Z_{p-1}$. Thus ${\cal K}\subseteq {\cal C}_{\emptyset}$.

Finally, the numbers of conjugacy classes in ${\cal C}_{\emptyset}$ and
in ${\cal K}$ are the same, since both of these sets can be labelled
by the $(p-1)$-tuples of partitions of $w$. We therefore deduce from
Proposition \ref{propsemidirect} that the irreducible characters
of $G_{p,w}$ with ${\cal H}$ in their kernel form a ${\cal
C}_{\emptyset}$-basic set for $G_{p,w}$.

Now, as we remarked when we presented the irreducible characters of
$G_{p,w}$, if $\alpha=(\alpha^1, \, \ldots , \, \alpha^p) \Vdash
w$ is such that $\alpha^r= \emptyset$, then $\chi^{\alpha} \in \Irr(G_{p,w})$ has $ {\cal H}$ in its kernel. Since these characters can
naturally be labelled by the $(p-1)$-tuples of partitions of $w$, they
must be all the irreducible characters of $G_{p,w}$ with ${\cal
H}$ in their kernel. This ends the proof.

\end{proof}

\section{A $p$-basic set for $\A_n$}\label{part5}

\subsection{Proof of Theorem~\ref{BSn}}
We keep the notation as above. As previously, for any positive integer
$w$ and odd prime $p$, we set $G_{p,w}=(\Z_p\rtimes \Z_{p-1})\wr
\sym_w$.
Using Lemma~\ref{blocks} we first reduce the problem as follows.
The signature $\varepsilon$ induces a permutation of order $2$ on the set
$\operatorname{Bk}_p(\sym_n)$ of $p$-blocks of $\sym_n$. Each orbit has one or two $p$-blocks
(which then have the same weight). Suppose that the orbit
$\omega=\{b,\varepsilon b\}$ has two $p$-blocks of weight $w \geq 0$. If
$w=0$, then $b$ and $\varepsilon b$ each consist of a single character
(which vanishes on $p$-singular elements), so that $b \cup \varepsilon
b$ is a $p$-basic set for itself, which satisfies Condition (2)
of Theorem~\ref{BSn} by Remark \ref{remw=0}. Suppose then that $w \neq 0$. By
Theorem~\ref{perfisom} there is a generalized perfect isometry
$$\cal I_{\eta} \colon (\Irr(G_{p,w}),\cal C_{\emptyset})\rightarrow
(b,p\textrm{-reg})$$
($\cal I_{\eta}$ is the inverse of the isometry $\cal J_{\eta}$ given by
Theorem~\ref{perfisom}). Furthermore, by Theorem~\ref{basicsetwreath},
we know that $G_{p,w}$ has a $\cal C_{\emptyset}$-basic set
$B_{\emptyset}$.
Hence we deduce, using
Proposition~\ref{BSisometrie}, that $b$ has a $p$-basic set
$B_{\emptyset}'$.
Consider now the set $\varepsilon B_{\emptyset}'\subseteq \varepsilon b$. 
This is clearly a
$p$-basic set of $\varepsilon b$. Moreover, by Lemma~\ref{blocks}, the
set $B_{\emptyset}'\cup\varepsilon B_{\emptyset}'$ 
is a $p$-basic set of $b\cup \varepsilon
b$, and it is $\varepsilon$-stable by construction. We now remark that,
if $\chi_{\lambda}\in b\cup \varepsilon b$, then $\lambda$ is never
self-conjugate, because the $p$-core of a self-conjugate partition
is also self-conjugate. Hence Condition (2) of Theorem~\ref{BSn} is
automatically satisfied by $B_{\emptyset}'\cup\varepsilon
B_{\emptyset}'$.

We can proceed as above for every orbit of size two, and we therefore
obtain a $p$-basic set of
$$\bigcup_{\omega=\{b,\varepsilon b\},\,|\omega|=2}
b\cup\varepsilon b,$$
which satisfies the conditions of Theorem~\ref{BSn}. Hence to prove
Theorem~\ref{BSn}, it is sufficient to prove that every
$\varepsilon$-stable $p$-block of $\sym_n$ has a $p$-basic set
satisfying the conditions of Theorem~\ref{BSn}.

Let $b$ be an $\varepsilon$-stable $p$-block of $\sym_n$ of weight $w
\geq 0$. As above, if $w=0$, then $b$ is a $p$-basic set for itself, and
satisfies Condition (2) of
Theorem~\ref{BSn} by Remark \ref{remw=0}. We therefore suppose that $w \neq 0$. As in the
previous case, write $\cal I_{\eta} \colon (\Irr(G_{p,w}),\cal
C_{\emptyset})\rightarrow (b,p\textrm{-reg})$ for the (inverse of the)
generalized perfect isometry given by
Theorem~\ref{perfisom}. Let
be the $\cal C_{\emptyset}$-basic set of $G_{p,w}$ obtained
in Theorem~\ref{basicsetwreath}. By Proposition~\ref{BSisometrie},
$B_{\emptyset}'=\cal I (B_{\emptyset})$ is a $p$-basic set of $b$. It
follows from the definition of $\cal I$ that
$$B_{\emptyset}'=\{\chi_{\lambda}\in b\ |\ \alpha_{\lambda}=(\alpha^1, \,
\ldots , \, \alpha^p) \Vdash w, \, \alpha^{(p+1)/2}=
\emptyset\}.$$
Let $\chi_{\lambda}\in B_{\emptyset}'$. Then, by
Lemma~\ref{conjugationabacus} and Convention~\ref{convention}, the
$r$-th part ($r = (p+1)/2$) of the $p$-quotient of $\lambda^*$
is empty. Hence $\varepsilon \chi_{\lambda}=\chi_{\lambda^*}\in
B_{\emptyset}'$. This proves that $B_{\emptyset}'$ is
$\varepsilon$-stable.
Lemma~\ref{preg} immediately implies that $B_{\emptyset}'$ satisfies
Condition (2) of Theorem~\ref{BSn}.
This proves the result.

\subsection{A $p$-basic set for $\A_n$}
We first recall the parametrizations of the classes and
irreducible characters of $\A_n$. Because of Clifford's
theory~\cite[III.2.12]{Feit}, we know that the set of
irreducible characters of $\A_n$ consists of the irreducible
constituents occurring in the restrictions to $\A_n$ of irreducible
characters of $\sym_n$. More precisely, let $\chi_{\lambda}$ be an
irreducible character of $\sym_n$. If $\lambda\neq\lambda^*$ then
$\Res_{\A_n}^{\sym_n}(\chi_{\lambda})=\Res_{\A_n}^{\sym_n}(\chi_{\lambda
^*})$ is irreducible. We will denote by $\rho_{\lambda}$ this
character of $\Irr(\A_n)$. If $\lambda$ is self-conjugate, then
$\Res_{\A_n}^{\sym_n}(\chi_{\lambda})$ splits into two distinct
constituents, denoted by $\rho_{\lambda,+}$ and $\rho_{\lambda,-}$. It
follows immediately from Clifford's theory~\cite[III.2.14]{Feit} that
$$
\Ind_{\A_n}^{\sym_n}(\rho_{\lambda})=\chi_{\lambda}+\chi_{\lambda^*}
\quad\textrm{and}\quad
\Ind_{\A_n}^{\sym_n}(\rho_{\lambda,+})=
\Ind_{\A_n}^{\sym_n}(\rho_{\lambda,+})=\chi_{\lambda}.
$$
Recall that the conjugacy classes of $\sym_n$ are parametrized
by the partitions of $n$ (the parts of a partition $\lambda$
of $n$ give the cycle structure of the corresponding conjugacy
class of $\sym_n$). If $\lambda$ is a self-conjugate partition
of $n$, then the class of $\sym_n$ labelled by the partition
$\overline{\lambda}$ is contained in $\A_n$, and splits into two
classes, $\overline{\lambda}_{+}$ and $\overline{\lambda}_{-}$, of
$\A_n$. Moreover if $c$ denotes a class of $\A_n$ distinct from
$\overline{\lambda}_{+}$ and $\overline{\lambda}_{-}$, then for $x\in
c$, we have (cf~\cite[2.5.13]{James-Kerber})
$$\rho_{\lambda,\pm}(x)=\chi_{\lambda}(x)/2,\quad
\rho_{\lambda,\pm}(\overline{\lambda}_{+})=x_{\lambda}\pm y_{\lambda}
\quad\textrm{and}\quad
\rho_{\lambda,\pm}(\overline{\lambda}_{-})=x_{\lambda}\mp y_{\lambda},$$
with $x_{\lambda}=\pm 1/2$ and $y_{\lambda}$ is some non-zero complex
number.

\begin{remark}\label{valentiere}
Note that if $x$ is a representative of a class of $\A_n$ distinct
from $\overline{\lambda}_{+}$ and $\overline{\lambda}_{-}$,
then $\rho_{\lambda,\pm}(x)$ is a rational number,
because $\rho_{\lambda,\pm}(x)=\chi_{\lambda}(x)/2$ with
$\chi_{\lambda}(x)\in\Z$. However, $\rho_{\lambda,\pm}(x)$ is an
algebraic integer. Hence it follows that $\rho_{\lambda,\pm}(x)\in\Z$.
\end{remark}

We can now prove:

\begin{theorem}\label{BSAn}
Let $n$ be a positive integer and $p$ be an odd prime. Then
$\A_n$ has a $p$-basic set.
\end{theorem}

\begin{proof}
Let $\B_{\emptyset}$ be the $p$-basic set of $\sym_n$ obtained in
Theorem~\ref{BSn}. We will prove that $\B_{\emptyset,\A_n}$, the
restriction of $\B_{\emptyset}$ to $\A_n$, 
is a $p$-basic set of
$\A_n$. We denote by $A$ the set of self-conjugate partitions
$\lambda$ of $n$ such that $\overline{\lambda}$ is $p$-regular, and by
$S$ the set of sets $\{\lambda,\lambda^*\}$ with $\lambda\neq\lambda^*$
and $\chi_{\lambda}\in B$. For $x\in S$ we set
$x=\{\lambda_x,\lambda_x^*\}$. By Lemma~\ref{preg}, if $\lambda=\lambda^*$, then $\chi_{\lambda} \in B$ if and only if $\lambda \in A$. We thus have
$$\B_{\emptyset}=\{\chi_{\lambda}\ |\ \lambda\in A\}\cup
\{\chi_{\lambda_x},\chi_{\lambda_x^*}\ |\ x\in S\}.$$
It therefore follows that
$$\B_{\emptyset,\A_n}=\{\rho_{\lambda,\pm}\ |\ \lambda\in A\} \cup
\{\rho_{\lambda_x}\ | \ x \in S\}.$$

We will now prove that the family $\B_{\emptyset,\A_n}^{p\textrm{-reg}}$ 
is free.
Let $a_{\lambda,\pm}$ ($\lambda\in A$), $a_x$ ($x \in S$) be integers
such that
\begin{equation}\label{rel1}
\sum_{\lambda\in
A}(a_{\lambda,+}\rho_{\lambda,+}^{p\textrm{-reg}}+a_{\lambda,-}
\rho_{\lambda,-}^{p\textrm{-reg}})+\sum_{x\in
S}a_x\rho_{\lambda_x}^{p\textrm{-reg}}=0.
\end{equation}
For every class function $\varphi$ of $\A_n$, we have
$\Ind_{\A_n}^{\sym_n}(\varphi^{p\textrm{-reg}})=(\Ind_{\A_n}^{\sym_n}(\varphi))^{p\textrm{-reg}}$.
Hence, inducing to $\sym_n$ Relation~(\ref{rel1}), we obtain
$$\sum_{\lambda\in
A}(a_{\lambda,+}+a_{\lambda,-})\chi_{\lambda}^{p\textrm{-reg}}
+\sum_{x\in
S}a_x(\chi_{\lambda_x}^{p\textrm{-reg}}+\chi_{\lambda_x^*}^{p\textrm{-reg}})=0.$$
Using the fact that $\B_{\emptyset}^{p\textrm{-reg}}$ is free, 
we deduce that $a_x=0$
for all $x \in S$
and $a_{\lambda,+}=-a_{\lambda,-}$ for all $\lambda \in A$. Therefore,
Relation~(\ref{rel1}) gives
$$
\sum_{\lambda\in
A}a_{\lambda,+}(\rho_{\lambda,+}^{p\textrm{-reg}}-
\rho_{\lambda,-}^{p\textrm{-reg}})
=0.
$$
We now fix $\lambda_0\in A$. The partition $\overline{\lambda}_0$ is
$p$-regular, so that the class of $\sym_n$ corresponding to
$\overline{\lambda}_0$ is $p$-regular. It follows that the classes
$\overline{\lambda}_{0,\pm}$ are also $p$-regular.
For $\lambda\in A$ distinct from
$\lambda_0$ we have
$\rho_{\lambda,+}^{p\textrm{-reg}}(\overline{\lambda}_{0,+})=
\rho_{\lambda,-}^{p\textrm{-reg}}(\overline{\lambda}_{0,+})$.
Moreover $\rho_{\lambda_0,\pm}^{p\textrm{-reg}}(\overline{\lambda}_{0,+})
=x_{\lambda_0}\pm y_{\lambda_0}$.
It follows that
$$2a_{\lambda,+}y_{\lambda_0}=0.$$
But $y_{\lambda_{0}}\neq 0$ implies that $a_{\lambda,+}=0$, so that
$a_{\lambda,+}=-a_{\lambda,-}=0$ for all $\lambda \in A$. Hence
$\B_{\emptyset,\A_n}^{p\textrm{-reg}}$ is free.

We will now prove that the family $\B_{\emptyset,\A_n}^{p\textrm{-reg}}$ 
generates
$\Irr(\A_n)^{p\textrm{-reg}}$ over $\Z$.
Let $\mu$ be a partition of $n$. Since $\B_{\emptyset}$ is a
$p$-basic set, we have
\begin{equation}\label{eq2}
\chi_{\mu}^{p\textrm{-reg}}=\sum_{\lambda\in
A}a_{\mu,\lambda}\chi_{\lambda}^{p\textrm{-reg}}
+\sum_{x\in
S}(a_{\mu,x} \chi_{\lambda_x}^{p\textrm{-reg}}+b_{\mu,x}
\chi_{\lambda_x^*}^{p\textrm{-reg}}),
\end{equation}
for some integers
$a_{\mu,\lambda}$ ($\lambda \in A$), $a_{\mu, x}$ and $b_{\mu x}$ ($x \in S$).

For every class function $\varphi$ of $\sym_n$, we
have $\Res_{\A_n}^{\sym_n} (\varphi^{p\textrm{-reg}})=
(\Res_{\A_n}^{\sym_n}(\varphi))^{p\textrm{-reg}}$.

We first suppose that $\mu\neq\mu^*$.
Hence restricting to
$\A_n$ Relation~(\ref{eq2}), we obtain
$$\rho_{\mu}^{p\textrm{-reg}}=\sum_{\lambda\in
A}a_{\mu,\lambda}(\rho_{\lambda,+}^{p\textrm{-reg}}+
\rho_{\lambda,-}^{p\textrm{-reg}})
+\sum_{x\in
S}(a_{\mu,x}+b_{\mu,x}) \rho_{\lambda_x}^{p\textrm{-reg}},$$ so that
$\rho_{\mu}^{p\textrm{-reg}}$ is a $\Z$-linear combination of elements
of $\B_{\emptyset,\A_n}^{p\textrm{-reg}}$.

We now suppose that $\mu$ is a self-conjugate partition not lying in
$A$. In Relation~(\ref{eq2}), we have $a_{\mu,x}=b_{\mu,x}$ because
$\chi_{\mu}$ is $\varepsilon$-stable.
Moreover $\overline{\mu}$ labels a $p$-singular class of
$\sym_n$. Thus
$\rho_{\mu,+}^{p\textrm{-reg}}=\rho_{\mu,-}^{p\textrm{-reg}}$, and it
follows that
$$2\rho_{\mu}^{p\textrm{-reg}}=\sum_{\lambda\in
A}a_{\mu,\lambda}(\rho_{\lambda,+}^{p\textrm{-reg}}+
\rho_{\lambda,-}^{p\textrm{-reg}})
+\sum_{x\in
S}2a_{\mu,x} \rho_{\lambda_x}^{p\textrm{-reg}}.$$
Let $\lambda_0\in A$. Evaluating the preceding relation on the class
$\overline{\lambda}_{0,+}$, we obtain
$$2x_{\lambda_0}a_{\mu,\lambda_0}=2\rho_{\mu}^{p\textrm{-reg}}(
\overline{\lambda}_{0,+})-
\sum_{\lambda\neq \lambda_0
}2a_{\mu,\lambda}\,\rho_{\lambda,+}^{p\textrm{-reg}}(\overline{\lambda}_{0,+})
-\sum_{x\in S}
2a_{\mu,x}\,
\rho_{\lambda_x}^{p\textrm{-reg}}(\overline{\lambda}_{0,+})\in2\Z,$$
because $\rho_{\mu}^{p\textrm{-reg}}(
\overline{\lambda}_{0,+})\in\Z$ and
$\rho_{\lambda,+}^{p\textrm{-reg}}(\overline{\lambda}_{0,+})\in\Z$ for
$\lambda\neq\lambda_0$ by Remark~\ref{valentiere}.
However, $2x_{\lambda_0}=\pm 1$, so that $a_{\mu,\lambda_0}=2a_{\mu,\lambda}'$ with
$a_{\mu,\lambda}'\in\Z$. We therefore deduce that
$$\rho_{\mu}^{p\textrm{-reg}}=\sum_{\lambda\in
A}a_{\mu,\lambda}'(\rho_{\lambda,+}^{p\textrm{-reg}}+
\rho_{\lambda,-}^{p\textrm{-reg}})
+\sum_{x\in
S}a_{\mu,x} \rho_{\lambda_x}^{p\textrm{-reg}}.$$
Since the result is obvious if $\mu \in A$, the result follows.
\end{proof}

\section{Consequences for the decomposition matrices of
$\A_n$}\label{consequence}

Let $G$ be a finite group and $p$ a prime. We denote by $\Br(G)$ the set
of irreducible Brauer characters of $G$ as in~\cite[IV.2]{Feit}. Then,
for every $\chi\in\Irr(G)$ and $\theta\in\Br(G)$, there are non-negative
integers $d_{\chi,\theta}$ (which are uniquely determined), such that
$$\chi^{p\textrm{-reg}}=\sum_{\theta\in\Br(G)}d_{\chi,\theta}\,
\theta.$$
The matrix $D=(d_{\chi,\theta})_{\chi\in\Irr(G),\theta\in\Br(G)}$ is the
so-called $p${\emph{-decomposition matrix}} of $G$~\cite[I.17]{Feit}. In
this section, we propose to deduce some results on the $p$-decomposition
matrix of $\A_n$ knowing that of $\sym_n$.

\subsection{Reduction of the problem}

We keep the notation as above. Let $B$ be a $p$-basic set of $G$.
We set $D_B=(d_{\chi,\theta})_{\chi\in B,\theta\in \Br(G)}$. Then
$D_B$ is a unimodular matrix. Moreover, $D_B$ determines $D$, because
we have 
$$D=P_BD_B,$$ 
where, $P_B$ is the matrix with coefficients $p_{\chi,\varphi}$
($\chi\in\Irr(G),\,\varphi\in B$), such that
$$\chi^{p\textrm{-reg}}=\sum_{\varphi\in
B}p_{\chi,\varphi}\,\varphi^{p\textrm{-reg}}.$$
Note that, if $G = \sym_n$ and $B=\B_{\emptyset}$, then we know
the coefficients $p_{\chi,\varphi}$ in the matrix $P_B$, at least
theoretically. They come from their analogues $a_{\chi,\varphi}$ in
$G_{p,w}$, described in the proof of Proposition \ref{propsemidirect}.

\begin{proposition}\label{bijection}
Let $p$ be an odd prime. Let $B$ be an $\varepsilon$-stable $p$-basic
set of $\sym_n$, where $\varepsilon$ denotes the sign character of
$\sym_n$. We denote by $D_B$ the restriction of the $p$-decomposition
matrix of $\sym_n$ to $B$ as above. Then $\varepsilon$ acts on
$\Br(\sym_n)$, the rows of $D_B$ and the columns of $D_B$, and these
three operations are equivalent.
\end{proposition}

\begin{proof}
We denote by $\Fix_{\varepsilon}(B)$ (respectively
$\Fix_{\varepsilon}(\Br(\sym_n))$) the set of $\varepsilon$-stable
characters in $B$ (respectively in $\Br(\sym_n)$). In order to prove the
assumption, we have to prove
$$|\Fix_{\varepsilon}(B)|=|\Fix_{\varepsilon}(\Br(\sym_n))|.$$
For ease of notation, we will, throughout the proof, write
$\hat{\varphi}$ for $ \varphi^{p\textrm{-reg}}$ ($\varphi \in
\Irr(\sym_n)$).

We set $D_B=(d_{\chi,\theta})_{\chi\in B,\theta\in\Br(\sym_n)}$
and $D_{B,\varepsilon}=(d_{\varepsilon\chi,
\hat{\varepsilon}\theta})_{\varepsilon\chi\in
B,\hat{\varepsilon}\theta\in\Br(\sym_n)}$. We have
$$\widehat{\varepsilon\chi}=\sum_{\theta\in\Br(\sym_n)}
d_{\varepsilon\chi,\theta}\theta.$$
Hence, it follows that
$$\begin{array}{lcl}
\hat{\chi}&=&\sum\limits_{\theta\in\Br(\sym_n)}d_{\varepsilon\chi,\theta}\hat{
\varepsilon}\theta\\
&=&
\sum\limits_{\hat{\varepsilon}\theta'\in\Br(\sym_n)}d_{\varepsilon\chi,\hat{
\varepsilon}\theta'}\theta'\\
&=&\sum\limits_{\theta'\in\Br(\sym_n)}d_{\varepsilon\chi,\hat{
\varepsilon}\theta'}\theta'.
\end{array}
$$
Since $\Br(\sym_n)$ is a $\Z$-basis of $\Z\Br(\sym_n)$, we deduce that
$d_{\chi,\theta}=d_{\varepsilon\chi,\hat{\varepsilon}\theta}$, that is
$$D_B=D_{B,\varepsilon}.$$
Furthermore, $D_{B,\varepsilon}$ is obtained from $D_B$ by a permutation
of the rows and of the columns. Thus, there are permutation matrices $P$
and $Q$ such that
$$D_{B,\varepsilon}=PD_BQ.$$  
We deduce that $D_B=PD_BQ$. Furthermore, we have $|B|=|\Br(\sym_n)|=m$.
Since $B$ is a $\Z$-basis of $\Z\Br(\sym_n)$, we have
$D_B\in\operatorname{GL}_m(\Z)$. Thus, we have
$$Q^{-1}=D_B^{-1}PD_B.$$  
It follows that $\Tr(Q^{-1})=\Tr(P)$. Since $P$, $Q$, $P^{-1}$
and $Q^{-1}$ are permutation matrices, we deduce that they
have the same trace. Furthermore, the number of fixed-point
of the permutation induced by $P$ is the number of fixed
rows. Hence, we have $\Tr(P)=|\Fix_{\varepsilon}(B)|$.
Moreover, we have $D_B^{-1}=QD_B^{-1}P$. We then deduce that
$\Tr(Q)=|\Fix_{\varepsilon}(\Br(\sym_n))|$, and the result follows.
\end{proof}

We denote by $\B_{\emptyset}$ the $\varepsilon$-stable $p$-basic set of
$\sym_n$ obtained in Theorem~\ref{BSn} and by $\B_{\A_n,\emptyset}$ the
restriction of $B_{\emptyset}$ to $\A_n$. As above, if $\lambda$ is a
partition of $n$, we denote by $\alpha_{\lambda}$ its $p$-quotient. As
above, we set
$$\Lambda_{\emptyset}=\{\lambda\vdash n\ |\
\alpha_{\lambda}^{(p-1)/2}=\emptyset\}.$$
We denote by $\Lambda_{\emptyset,1}$ the set of
$\lambda\in\Lambda_{\emptyset}$ such that $\lambda=\lambda^*$ and
by $\Lambda_{\emptyset,2}$ the set of sets $\{\lambda,\lambda^*\}$
for $\lambda\in\Lambda_{\emptyset}$ with $\lambda\neq\lambda^*$. For
$w\in\Lambda_{\emptyset,2}$, we fix $\lambda_{w}\in w$. In particular,
we have $w=\{\lambda_w,\lambda_w^*\}$. In Theorem~\ref{BSn} and in the
proof of Theorem~\ref{BSAn}, we proved that
$$\B_{\emptyset}=\{\chi_{\lambda}\ |\
\lambda\in\Lambda_{\emptyset}\}\quad\textrm{and}\quad
\B_{\A_n,\emptyset}=\{\rho_{\lambda,\pm}\ |\ \lambda \in
\Lambda_{\emptyset,1}\}\cup\{\rho_{\lambda_w}\ |\
w\in\Lambda_{\emptyset,2}\},$$
where
$\Res_{\A_n}^{\sym_n}(\chi_{\lambda})=\rho_{\lambda,+}+\rho_{\lambda,-}$
for $\lambda\in\Lambda_{\emptyset,1}$, and $\rho_{\lambda_w}=
\Res_{\A_n}^{\sym_n}(\chi_{\lambda_w})$ for $w\in\Lambda_{\emptyset,2}$.

By Proposition \ref{bijection}, we can take
$\Phi:\B_{\emptyset} \rightarrow\Br(\sym_n)$ an
$\varepsilon$-equivariant bijection. For $\lambda, \, \mu
\in\Lambda_{\emptyset}$, we set $\theta_{\mu}=\Phi(\chi_{\mu})$ and
$d_{\lambda,\mu}=d_{\chi_{\lambda},\theta_{\mu}}$. By Clifford's
modular theory~\cite[9.18]{Hupp} (which could be applied because
the quotient $\sym_n/\A_n$ is cyclic of order prime to $p$) we
know how the $\theta_{\mu}$'s restrict from $\sym_n$ to $\A_n$. If
$\mu\in\Lambda_{\emptyset,1}$, then $\Res_{\A_n}^{\sym_n}(\theta_{\mu})$
is the sum of two irreducible Brauer characters of $\A_n$,
which we denote by $\theta_{\mu,+}'$ and $\theta_{\mu,-}'$.
If $w\in\Lambda_{\emptyset,2}$, then $\theta_{\mu_w}'=
\Res_{\A_n}^{\sym_n}(\theta_{\mu_w}) \in \Br(\A_n)$. We then have
$$\Br(\A_n)=\{\theta_{\mu,\pm}'\ |\
\mu\in\Lambda_{\emptyset,1}\}\cup\{\theta_{\mu_w}'\ |\
w\in\Lambda_{\emptyset,2}\}.$$
For $\lambda, \, \mu \in\Lambda_{\emptyset,1}$ and $w, \, w'
\in\Lambda_{\emptyset,2}$, there are non-negative integers $d_{\lambda;
\pm,\mu_w}'$, $d_{\lambda;\pm,\mu , \pm}'$, $d_{\lambda_{w'},\mu_w}'$
and $d_{\lambda_{w'},\mu;\pm}'$ such that
$$\rho_{\lambda,\pm}^{p\textrm{-reg}}=\sum_{\mu\in\Lambda_{\emptyset,1}}
(d_{\lambda; \pm,\mu;+}'\theta_{\mu,+}'+d_{\lambda;\pm,\mu;-}'\theta_{\mu,-}')
+\sum_{w\in\Lambda_{\emptyset,2}}d_{\lambda;\pm,\mu_{w}}'\theta_{\mu_w},$$
$$\rho_{\lambda_{w'}}^{p\textrm{-reg}}=\sum_{\mu\in\Lambda_{\emptyset,1}}
(d_{\lambda_{w'},\mu;+}'\theta_{\mu,+}'+d_{\lambda_{w'},\mu;-}'
\theta_{\mu,-}')+\sum_{w\in\Lambda_{\emptyset,2}}d_{\lambda_{w'},\mu_{w}}'\theta_{\mu_w}.$$
We then obtain the following about the $p$-decomposition numbers of
$\sym_n$ and $\A_n$:

\begin{theorem}\label{nbdecAn}
With the above notations, we have, for all $\lambda,
\, \mu \in\Lambda_{\emptyset,1}$ and for all $w, \, w'
\in\Lambda_{\emptyset,2}$,

\noindent
(i)$$
d_{\lambda,\mu_{w}}=d_{\lambda,\mu_w^*}
$$
$$
d_{\lambda_{w'},\mu}=d_{\lambda_{w'}^*,\mu}
$$
$$
d_{\lambda_{w'},\mu_w}=d_{\lambda_{w'}^*,\mu_w^*} \; \; \mbox{and} \; \;
d_{\lambda_{w'},\mu_w^*}=d_{\lambda_{w'}^*,\mu_w},
$$

\noindent
(ii)$$
d'_{\lambda;+,\mu;+}+d'_{\lambda;-,\mu;+}= d_{\lambda,\mu}=d'_{\lambda;+,\mu;-}+d'_{\lambda;-,\mu;-}
$$
$$
d'_{\lambda;+,\mu;+}=d'_{\lambda;-,\mu;-} \; \; \mbox{and} \; \; d'_{\lambda;+,\mu;-}=d'_{\lambda;-,\mu;+}
$$
$$
d'_{\lambda;+,\mu_w}=d'_{\lambda;-,\mu_w}=d_{\lambda,\mu_{w}},
$$

\noindent
(iii)$$d_{\lambda_{w'},\mu_w}+d_{\lambda_{w'},\mu_w^*}=d'_{\lambda_{w'},\mu_{w}}$$
$$d'_{\lambda_{w'},\mu;+}=d'_{\lambda_{w'},\mu;-}=d_{\lambda_{w'},\mu}.$$

\end{theorem}

\begin{proof}

\noindent
(i) This is a direct consequence of our parametrizations and our choice
of $\Phi$.

\noindent
(ii) Since
$\Res_{\A_n}^{\sym_n}(\chi_{\lambda})=\rho_{\lambda,+}+\rho_{\lambda,-}$
, $\Res_{\A_n}^{\sym_n}(\theta_{\mu})=\theta_{\mu,+}'+\theta_{\mu,-}'$
and
$\Res_{\A_n}^{\sym_n}(\theta_{\mu_w})=\Res_{\A_n}^{\sym_n}(\theta_{\mu_w
^*})=\theta_{\mu_w}'$, we immediately get
$$d'_{\lambda;+,\mu;+}+d'_{\lambda;-,\mu;+}=
d_{\lambda,\mu}=d'_{\lambda;+,\mu;-}+d'_{\lambda;-,\mu;-}
\; \mbox{and} \;
d_{\lambda,\mu_w}+d_{\lambda,\mu_w^*}=d'_{\lambda;+,\mu_w}+d'_{\lambda;-
,\mu_w}.$$
Now, if we write $\tau$ for the automorphism of $\A_n$ induced by
conjugation by a transposition of $\sym_n$, then $\tau$ also acts on
$\Irr(\A_n)$ and $\Br(\A_n)$, and we have
$$ \tau(\theta_{\mu,+}')=\theta_{\mu,-}', \; \; \tau ( \theta_{\mu_w}') = \theta_{\mu_w}' \; \; \mbox{and} \; \;  \tau( \rho_{\lambda,+}) = \rho_{\lambda,-}.$$
This implies
$$d'_{\lambda;+,\mu;+}=d'_{\lambda;-,\mu;-} , \; \;  d'_{\lambda;+,\mu;-}=d'_{\lambda;-,\mu;+} \; \; \mbox{and} \; \; d'_{\lambda;+,\mu_w}=d'_{\lambda;-,\mu_w}.$$
Together with (i) and the above, this finally gives
$d'_{\lambda;+,\mu_w}=d'_{\lambda;-,\mu_w}=d_{\lambda,\mu_{w}}$.

\noindent
(iii) Since
$\Res_{\A_n}^{\sym_n}(\chi_{\lambda_{w'}})=\Res_{\A_n}^{\sym_n}(\chi_{
\lambda_{w'}^*})=\rho_{\lambda_{w'}}$,
$\Res_{\A_n}^{\sym_n}(\theta_{\mu})=\theta_{\mu,+}'+\theta_{\mu ,-}'$
and $ \Res_{\A_n}^{\sym_n}(\theta_{\mu_w})=\Res_{\A_n}^{\sym_n}(\theta_{
\mu_w^*})=\theta_{\mu_w}'$, the desired equalities follow easily.

\end{proof}

We now show on a picture how to use Theorem \ref{nbdecAn}
to obtain decomposition numbers of $\A_n$, knowing those of
$\sym_n$ (or conversely). We denote by $D_{\B_{\emptyset}}$
the restriction to the basic set $\B_{\emptyset}$ of the
$p$-decomposition matrix of $\sym_n$, and $D_{\B_{\emptyset,\A_n}}$
the restriction to $\B_{\emptyset,\A_n}$ of the decomposition matrix
of $\A_n$. Using the notations of Theorem \ref{nbdecAn}, we write
$a=d_{\lambda,\mu}$, $b=d_{\lambda,\mu_{w}}$, $c=d_{\lambda_{w'},\mu}$,
$d=d_{\lambda_{w'},\mu_w}$, $e=d_{\lambda_{w'},\mu_w^*}$,
$\alpha=d'_{\lambda;+,\mu;+}$ and $\beta=d'_{\lambda;+,\mu;-}$. We also
write $a_1=|\Lambda_{\emptyset,1}|$ and $a_2=|\Lambda_{\emptyset,2}|$.
We start, at the level of $\sym_n$, with

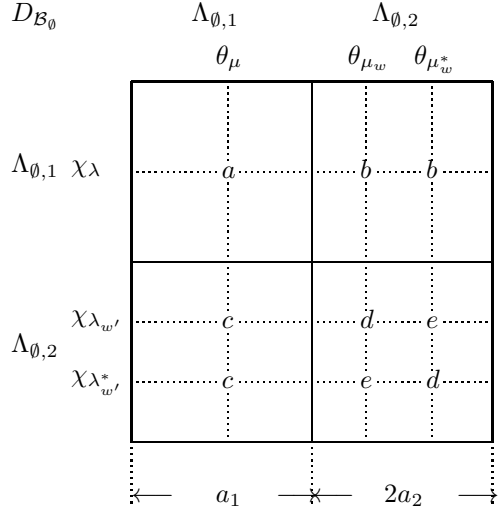
\setlength{\unitlength}{0.8mm}

\begin{center}

\begin{picture}(100,95)

\put(20,20) {\line(0,1){60}}

\put(50,20) {\line(0,1){60}}

\put(80,20) {\line(0,1){60}}

\put(20,20) {\line(1,0){60}}

\put(20,50) {\line(1,0){60}}

\put(20,80) {\line(1,0){60}}

\dottedline{1}(20,20)(20,10)

\put(20,10){$\longleftarrow$}

\put(44,10){$\longrightarrow$}

\put(50,10){$\longleftarrow$}

\put(74,10){$\longrightarrow$}

\put(34,10){$a_1$}

\put(62,10){$2a_2$}

\dottedline{1}(50,20)(50,10)

\dottedline{1}(80,20)(80,10)

\put(0,35){$\Lambda_{\emptyset,2}$}

\put(0,65){$\Lambda_{\emptyset,1}$}

\put(60,90){$\Lambda_{\emptyset,2}$}

\put(30,90){$\Lambda_{\emptyset,1}$}

\put(0,90){$D_{\B_{\emptyset}}$}

\put(10,30){$\chi_{\lambda_{w'}^*}$}

\put(10,40){$\chi_{\lambda_{w'}}$}

\put(10,65){$\chi_{\lambda}$}

\put(34,83){$\theta_{\mu}$}

\put(56,83){$\theta_{\mu_w}$}

\put(67,83){$\theta_{\mu_w^*}$}

\dottedline{1}(20,65)(34,65)

\dottedline{1}(20,40)(34,40)

\dottedline{1}(20,30)(34,30)

\dottedline{1}(38,65)(57,65)

\dottedline{1}(38,40)(57,40)

\dottedline{1}(38,30)(57,30)

\dottedline{1}(61,65)(68,65)

\dottedline{1}(61,40)(68,40)

\dottedline{1}(61,30)(68,30)

\dottedline{1}(72,65)(80,65)

\dottedline{1}(72,40)(80,40)

\dottedline{1}(72,30)(80,30)

\put(35,64){$a$}

\put(35,39){$c$}

\put(35,29){$c$}

\put(58,64){$b$}

\put(58,39){$d$}

\put(58,29){$e$}

\put(69,64){$b$}

\put(69,39){$e$}

\put(69,29){$d$}

\dottedline{1}(36,80)(36,67)

\dottedline{1}(36,63)(36,42)

\dottedline{1}(36,38)(36,32)

\dottedline{1}(36,28)(36,20)

\dottedline{1}(59,80)(59,67)

\dottedline{1}(59,63)(59,42)

\dottedline{1}(59,38)(59,32)

\dottedline{1}(59,28)(59,20)

\dottedline{1}(70,80)(70,67)

\dottedline{1}(70,63)(70,42)

\dottedline{1}(70,38)(70,32)

\dottedline{1}(70,28)(70,20)

\end{picture}

\end{center}

Then, by Theorem \ref{nbdecAn}, we obtain, at the level of $\A_n$,

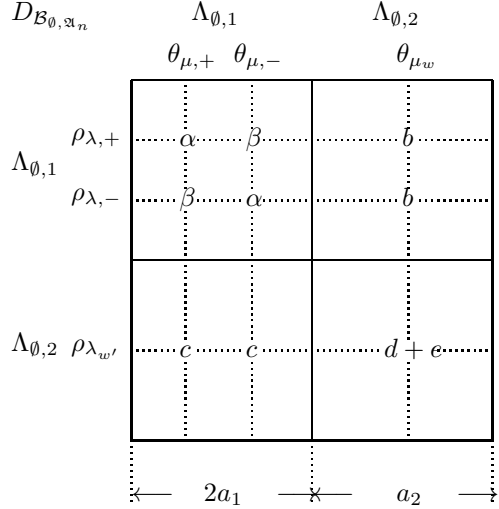
\setlength{\unitlength}{0.8mm}

\begin{center}

\begin{picture}(100,95)

\put(20,20) {\line(0,1){60}}

\put(50,20) {\line(0,1){60}}

\put(80,20) {\line(0,1){60}}

\put(20,20) {\line(1,0){60}}

\put(20,50) {\line(1,0){60}}

\put(20,80) {\line(1,0){60}}

\dottedline{1}(20,20)(20,10)

\put(20,10){$\longleftarrow$}

\put(44,10){$\longrightarrow$}

\put(50,10){$\longleftarrow$}

\put(74,10){$\longrightarrow$}

\put(32,10){$2a_1$}

\put(64,10){$a_2$}

\dottedline{1}(50,20)(50,10)

\dottedline{1}(80,20)(80,10)

\put(0,35){$\Lambda_{\emptyset,2}$}

\put(0,65){$\Lambda_{\emptyset,1}$}

\put(60,90){$\Lambda_{\emptyset,2}$}

\put(30,90){$\Lambda_{\emptyset,1}$}

\put(0,90){$D_{\B_{\emptyset,\A_n}}$}

\put(10,60){$\rho_{\lambda,-}$}

\put(10,70){$\rho_{\lambda,+}$}

\put(10,35){$\rho_{\lambda_{w'}}$}

\put(64,83){$\theta_{\mu_w}$}

\put(26,83){$\theta_{\mu,+}$}

\put(37,83){$\theta_{\mu,-}$}

\dottedline{1}(20,35)(27,35)

\dottedline{1}(20,70)(27,70)

\dottedline{1}(20,60)(27,60)

\dottedline{1}(31,35)(38,35)

\dottedline{1}(31,70)(38,70)

\dottedline{1}(31,60)(38,60)

\dottedline{1}(42,35)(61,35)

\dottedline{1}(42,70)(64,70)

\dottedline{1}(42,60)(64,60)

\dottedline{1}(71,35)(80,35)

\dottedline{1}(68,70)(80,70)

\dottedline{1}(68,60)(80,60)

\put(62,34){$d+e$}

\put(65,69){$b$}

\put(65,59){$b$}

\put(28,34){$c$}

\put(28,69){$\alpha$}

\put(28,59){$\beta$}

\put(39,34){$c$}

\put(39,69){$\beta$}

\put(39,59){$\alpha$}

\dottedline{1}(66,80)(66,72)  

\dottedline{1}(66,68)(66,62)  

\dottedline{1}(66,58)(66,37)  

\dottedline{1}(66,33)(66,20)  

\dottedline{1}(29,80)(29,72)

\dottedline{1}(29,68)(29,62)

\dottedline{1}(29,58)(29,37)

\dottedline{1}(29,33)(29,20)

\dottedline{1}(40,80)(40,72)

\dottedline{1}(40,68)(40,62)

\dottedline{1}(40,58)(40,37)

\dottedline{1}(40,33)(40,20)

\end{picture}

\end{center}
Note that, by Theorem \ref{nbdecAn}, we have $\alpha + \beta = a$.

\subsection{Consequences}
In this section, we suppose that we know the decomposition matrix
$D$ of $\sym_n$. We denote by $D_{\B_\emptyset}$ the restriction of
$D$ to the $p$-basic set $\B_{\emptyset}$ as above. Let $D_{n,p}$
be the submatrix of $D_{\B_\emptyset}$ whose entries are all
the entries of $D_{\B_\emptyset}$ which lie at the intersection
of an $\varepsilon$-stable row and an $\varepsilon$-column of
$D_{\B_\emptyset}$. We can restrict this matrix to $\A_n$. We thus
obtain a matrix $D_{n,p}'$ which has twice as many rows and columns as
$D_{n,p}$.

\begin{theorem}\label{constmat}
We keep the notation as above. We suppose that the $p$-decomposition
matrix $D$ of $\sym_n$ and the matrix $D_{n,p}'$ are known. Then we can
construct the $p$-decomposition matrix of $\A_n$.
\end{theorem}

\begin{proof}
As above, we denote by $D_{\B_{\emptyset,\A_n}}$ the restriction
of the $p$-decomposition matrix of $\A_n$ to the $p$-basic set
$\B_{\emptyset,\A_n}$ (obtained in Theorem~\ref{BSAn}). As explained
at the begining of Section~\ref{consequence}, to compute the
$p$-decomposition matrix of $\A_n$, it is sufficient to compute the
matrix $D_{\B_{\emptyset,\A_n}}$ (again, provided we know the matrix
$P_{\B_{\emptyset,\A_n}}$). Furthermore, using Theorem~\ref{nbdecAn},
every coefficient of $D_{\B_{\emptyset,\A_n}}$ which does not lie in
$D_{n,p}'$ can be obtained from $D$. The result follows.
\end{proof}

The above theorem shows that, in order to deduce the $p$-decomposition
matrix of $\A_n$ from that of $\sym_n$, we just have to understand
the matrix $D_{n,p}'$, which is a small matrix compared to the
$p$-decomposition matrix of $\A_n$.

Following \cite[\S6.3]{James-Kerber}, we say that a matrix has wedge
shape if its rows and its columns can be ordered in such a way that the
resulting matrix is a lower triangular matrix with diagonal entries
equal to $1$.

\begin{lemma}
We suppose that $\sym_n$ has a $p$-basic set $B'$ satisfying the
properties of Theorem~\ref{BSn}, and such that the restriction matrix
$D_{B'}$ has wedge shape. Then the matrix $D_{n,p}$ defined above
(and obtained from the $p$-basic set $\B_{\emptyset}$ constructed in
Theorem~\ref{BSn}) has wedge shape.
\label{triangulaire}
\end{lemma}

\begin{proof}
Let $\chi_{\lambda}\in B'$ be such that $\lambda=\lambda^*$.
Since $B'$ has Property (2) of Theorem~\ref{BSn}, it follows
that $\overline{\lambda}$ is $p$-regular. Furthermore, by
Lemma~\ref{preg}, we see that $\lambda\in\Lambda_{\emptyset}$. Hence,
the $\varepsilon$-stable characters in $B'$ are the same as those in
$\B_{\emptyset}$. The result follows.
\end{proof}

\begin{remark} The $p$-basic set $\B_{\emptyset}$ does not have wedge
shape. The first counter-example is given by $\sym_6$ with $p=3$. In
this case, we have
$$D_{\B_{\emptyset}}=
\begin{bmatrix}
 1& 0& 0& 0& 0& 0& 0 \\
 0& 1& 0& 0& 0& 0& 0 \\
 1& 1& 1& 1& 1& 0& 0 \\
 0& 0& 0& 0& 0& 1& 0 \\
 0& 0& 0& 0& 0& 0& 1 \\
 0& 0& 1& 1& 0& 0& 0 \\
 0& 0& 1& 0& 1& 0& 0 
\end{bmatrix},$$
which has no wedge shape. However, there is, at the moment, no
example where Lemma~\ref{triangulaire} does not apply. Indeed, one
can check, using \textsc{Gap}~\cite{gap}, that, for $1\leq n\leq 18$
and for any odd prime, there always exists a $p$-basic set $B'$ as in
Lemma~\ref{triangulaire}. We can then conjecture that such a $p$-basic
set of $\sym_n$ always exists and that $D_{n,p}$ has wedge shape.
\end{remark}

\begin{lemma}
With the above notation, if $D_{n,p}$ has wedge shape, then so does
$D_{n,p}'$.
\label{triangulaireAn}
\end{lemma}

\begin{proof}We suppose that $D_{n,p}$ is a matrix of size $r$ and has
wedge shape. Hence we can order the rows and columns in such a way that
the resulting matrix is lower triangular with diagonal entries equal
to $1$. This induces permutations $\tau_r$ and $\tau_c$ on the rows
and the columns of $D_{n,p}$ respectively. By definition, $D_{n,p}'$
is a matrix of size $2r$. Now, we first regroup the rows and the
columns of $D_{n,p}'$ in corresponding pairs, according to Theorem
\ref{nbdecAn}, and label the pairs by $(2k-1,2k)$ for $1\leq k\leq r$.
We then apply the permutations of the rows and the columns defined by
$(2k-1,2k)\mapsto (2\tau_r(k)-1,2\tau_r(k))$ and $(2k-1,2k)\mapsto
(2\tau_c(k)-1,2\tau_c(k))$ respectively. The resulting matrix is thus
lower triangular by blocks, and the blocks on the diagonal have size
$2$. To every coefficient $d_{i,j}$ of $D_{n,p}$, we associate the
submatrix $D_{i,j}$ of $D'_{n,p}=(d_{i,j}')$ given by
$$D_{i,j}=
\begin{bmatrix}
d'_{2i-1,2j-1}&d'_{2i-1,2j}\\
d'_{2i,2j-1}&d'_{2i,2j}
\end{bmatrix}.$$
In particular, the 2 by 2 blocks on the diagonal of $D'_{n,p}$ are the
$D_{i,i}$'s, $1\leq i\leq r$. Moreover, Theorem~\ref{nbdecAn} asserts
that $d'_{2i-1,2j-1}=d'_{2i,2j}$, $d'_{2i-1,2j}=d'_{2i,2j-1}$ and
$d'_{2i-1,2j-1}+d'_{2i-1,2j}=d_{i,j}$. Since $d_{i,i}=1$, we thus deduce
that
$$D_{i,i}=
\begin{bmatrix}
1&0\\
0&1
\end{bmatrix}\quad\textrm{or}\quad
D_{i,i}=
\begin{bmatrix}
0&1\\
1&0
\end{bmatrix}.
$$
We can now re-order the rows of $D_{n,p}'$ as follows: for every $1\leq
i\leq r$, if $D_{i,i}$ is not the identity matrix, then we permute the
rows $2i-1$ and $2i$ of $D_{n,p}'$. The resulting matrix has wedge
shape, as required.
\end{proof}

\begin{remark}
Note that, if $D_{n,p}$ has wedge shape, then, using our $p$-basic set
$\B_{\emptyset,\A_n}$ obtained in Theorem~\ref{BSAn}, we have a way to
parametrize the simple module of $\A_n$ in characteristic $p$ which are
not the restriction of a simple module of $\sym_n$.
\end{remark}

\begin{remark}
We now discuss some problems and questions. We suppose that
the matrix $D_{n,p}$ is known and has size $r$. Let $d_{i,j}$
be the $(i,j)$-coefficient of $D_{n,p}$. As in the proof of
Lemma~\ref{triangulaireAn}, we can associate to this coefficient the 2
by 2 matrix $D_{i,j}$. Following Theorem~\ref{constmat}, to construct
the $p$-decomposition matrix of $\A_n$, it is sufficient to compute
all the matrices $D_{i,j}$ ($1\leq i,j\leq r$). Furthermore, by
Theorem~\ref{nbdecAn}, we have
$$D_{i,j}=
\begin{bmatrix}
a&b\\
b&a
\end{bmatrix},
$$
with $a+b=d_{i,j}$.
Hence, the computation of the $p$-decomposition matrix of $\A_n$ is
reduced to the following problems. \begin{itemize}
\item For $1\leq i,j\leq r$, we have to determine the integers $a$ and
$b$ appearing in the matrix $D_{i,j}$.
\item If the coefficients $a$ and $b$ appearing in all the matrices
$D_{i,j}$ ($1\leq i,j\leq r$) are known, then a $p$-decomposition matrix
of $\A_n$ is known, up to some choices. We have to develop a rule to fix
theses choices.

\end{itemize}
\end{remark}

\noindent\textbf{Acknowledgements.}\quad
Part of this work was done during a visit of the first author in
Lausanne. The first author gratefully acknowledges financial support by
the IGAT (EPFL). We wish to thank Michael Cuntz for many helpful and
valuable discussions on the subject. Some of the ideas in this paper
originated from discussions with him during a postdoctoral stay of the
first author at the University of Kaiserlautern. We also thank Meinolf
Geck for his reading of the manuscript, and for his suggestions and
helpful comments on this paper.

\bibliographystyle{plain}
\bibliography{referencesJB}

\end{document}